\documentclass[12pt]{amsart}
\usepackage{amssymb}
\usepackage{amsmath}
\usepackage{amscd}
\usepackage{verbatim}
\usepackage{amssymb,amsfonts,pstricks,epsf}
\usepackage{graphicx}
\usepackage{epsfig}
\usepackage{color}
\usepackage{stmaryrd}
\usepackage{mathrsfs}
\usepackage{marginnote}
\usepackage{geometry}
\newtheorem{theorem}{Theorem}[section]

\newtheorem{proposition}[theorem]{Proposition}
\newtheorem{corollary}[theorem]{Corollary}

\newtheorem{definition}[theorem]{Definition}

\definecolor{plum}{rgb}{1.0, 0.0, 1.0}

\DeclareMathOperator{\osc}{\textup{osc}}
\DeclareMathOperator{\esssup}{\textup{ess\ sup}}
\DeclareMathOperator{\essinf}{\textup{ess\ inf}}

\makeatletter
\@namedef{subjclassname@2020}{%
  \textup{2020} Mathematics Subject Classification}
\makeatother

\begin{document}

\title{PDE Comparison Principles for Robin Problems}

\author{Jeffrey J. Langford}

\address{Department of Mathematics, Bucknell University, Lewisburg, Pennsylvania 17837}

\email{jeffrey.langford@bucknell.edu}

\date{\today}

\begin{abstract}
We compare the solutions of two Poisson problems in a spherical shell with Robin boundary conditions, one with given data, and one where the data has been cap symmetrized. When the Robin parameters are nonnegative, we show that the solution to the symmetrized problem has larger convex means. Sending one of the Robin parameters to $+\infty$, we obtain mixed Robin/Dirichlet comparison results in shells. We prove similar results on balls and  prove a comparison principle on generalized cylinders with mixed Robin/Neumann boundary conditions.
\end{abstract}

 

\maketitle

\section{Introduction}

Comparison principles in elliptic pde began with the work of Talenti \cite{Talenti}. In his seminal work, Talenti compared the solutions of two Poisson problems with Dirichlet boundary conditions:
\[
\begin{array}{rclccccrclcc}
-\Delta u & = & f & \text{in} & \Omega, &  &  & -\Delta v & = & f^{\#} & \text{in} & \Omega^{\#},\\
u & = & 0 & \text{on} & \partial \Omega, &  &  &v & = & 0 & \text{on} & \partial \Omega^{\#}.
\end{array}
\]
In the first problem, the given data $f\geq0$ belongs to $L^2(\Omega)$ and $\Omega\subseteq \mathbb{R}^n$ denotes a bounded Lipschitz domain. In the second, symmetrized problem, $\Omega^{\#}\subseteq \mathbb{R}^n$ denotes a centered open ball with $|\Omega|=|\Omega^{\#}|$ (here $|\cdot|$ denotes Lebesgue measure), and $f^{\#}\in L^2(\Omega^{\#})$ denotes the Schwarz rearrangement, a radially decreasing function on $\Omega^{\#}$ whose upper level sets have the same measure as those of $f$, i.e. $|\{x\in \Omega:f(x)>t\}|=|\{x\in \Omega^{\#}:f^{\#}(x)>t\}|$ for each $t\in \mathbb{R}$. Talenti showed that the solutions $u$ and $v$ compare through their Schwarz rearrangements:
\begin{equation}\label{eq:uvcompare}
u^{\#}\leq v \quad \textup{in} \quad \Omega^{\#}.
\end{equation}
This conclusion implies, for instance,
\begin{equation}\label{eq:introconvexmeans}
\int_{\Omega}\phi(u)\,dx \leq \int_{\Omega^{\#}}\phi(v)\,dx
\end{equation}
for each increasing convex function $\phi:\mathbb{R} \to \mathbb{R}$. By the Minimum Principle, $u$ and $v$ are nonnegative, so choosing $\phi(x)=\chi_{[0,+\infty)}(x)\cdot x^p$ gives
\begin{equation}\label{eq:introLpnorms}
\|u\|_{L^p(\Omega)}\leq \|v\|_{L^p(\Omega^{\#})},\qquad 1\leq p\leq +\infty.
\end{equation}
The novelty of Talenti's result is that it allows us to estimate properties of the solution to a given pde in terms of the solution of a related problem that is, in general, much easier to solve. In particular, notice that the solution $v$ to the symmetrized problem is purely radial.

Papers on comparison principles related to Talenti's original work fill many journal pages. We direct the interested reader to the recent survey article, and references therein, by Talenti \cite{TalentiSurvey}. This literature is rich with pde comparison principles for problems with Dirichlet boundary conditions. We recall two results most relevant to our paper. In \cite{AlvinoDiaz Lions and Trombetti}, Alvino, Diaz, Lions, and Trombetti prove a comparison principle for Steiner symmetrization. To Steiner symmetrize a domain, we intersect the given domain with hyperplanes, and replace those intersections with centered balls of equal volume. To Steiner symmetrize a function, we apply the Schwarz rearrangement to each slice function. In this Steiner regime, the authors show that the solutions of the given and symmetrized problems compare via their convex means as in \eqref{eq:introconvexmeans}, with $\#$ denoting Steiner symmetrization. In \cite{Weitsman}, Weitsman establishes a comparison principle for cap symmetrization. Under cap symmetrization, we intersect a given domain with spheres, and replace those intersections with centered spherical caps of equal surface measure. Likewise, we cap symmetrize functions by applying the spherical symmetrization (the analogue of the Schwarz rearrangement for functions defined on spheres) to each radial slice function. In the cap regime, Weitsman shows that the solutions of the given and symmetrized problems compare as in \eqref{eq:uvcompare}, with $\#$ denoting cap symmetrization.

In contrast to the Dirichlet setting, there are few Neumann results that compare solutions to Poisson problems of the form
\[
\begin{array}{rclccccrclcc}
-\Delta u & = & f & \text{in} & \Omega, &  &  & -\Delta v & = & f^{\#} & \text{in} & \Omega^{\#},\\
\frac{\partial u}{\partial \nu} & = & 0 & \text{on} & \partial \Omega, &  &  &\frac{\partial v}{\partial \nu} & = & 0 & \text{on} & \partial \Omega^{\#},
\end{array}
\]
where $\frac{\partial u}{\partial \nu}$ is the outer normal derivative and $\#$ denotes a prescribed symmetrization. The Neumann problem presents a slight technicality in that the solutions $u$ and $v$ are not unique, so one must impose an appropriate normalization, say, that $u$ and $v$ have zero mean. For all known comparison results, the initial domain $\Omega$ and the symmetrized domain $\Omega^{\#}$ are equal. Thus, in moving from the first problem to the symmetrized one, it is only the input function $f$ that is rearranged. For example, in \cite{Langford1}, the author compares $u$ and $v$ when $\Omega$ is a spherical shell (the region between two concentric spheres) and  $\#$ is the cap symmetrization. The author shows that the solutions compare through their convex means as in \eqref{eq:introconvexmeans}. Baernstein, Laugesen, and the author prove a generalization of this result (see Theorem 10.20) in  \cite{Baernstein Symmetrization in Analysis}. See also \cite{Langford2} for a number of related comparison results on spheres and two-dimensional consequences in the plane via projection mappings. In \cite{Langford3}, the author compares $u$ and $v$ when $\Omega=D\times (0,\ell)$ is a generalized cylinder and $\#$ denotes the monotone decreasing rearrangement in the final variable (a one-dimensional Steiner symmetrization in one direction). Again, the author shows that the solutions compare via their convex means as in \eqref{eq:introconvexmeans}. In \cite{Periodic Steiner}, Brock considers similar cylinder-like sets and proves a related comparison principle with an eye towards comparing the extreme values ($\sup$ and $\inf$) of $u$ and $v$.

Likely driven by the recent surge of interest in Robin problems in spectral theory (see any of \cite{AFK}, \cite{Bareket}, \cite{Bossel}, \cite{BFNT}, \cite{BucurGiacomini1}, \cite{BucurGiacomini2}, \cite{ChasmanLangford}, \cite{Daners}, \cite{FK}, \cite{FreitasLaugesen1}, \cite{FrietasLaugesen2}, \cite{GirourdLaugesen}), in \cite{ANT}, Alvino, Nitsch, and Trombetti establish the first Robin comparison principle in the spirit of Talenti. The authors compare the solutions to the Poisson problems
\[
\begin{array}{rclccccrclcc}
-\Delta u & = & f & \text{in} & \Omega, &  &  & -\Delta v & = & f^{\#} & \text{in} & \Omega^{\#},\\
\frac{\partial u}{\partial \nu}+\alpha u & = & 0 & \text{on} & \partial \Omega, &  &  &\frac{\partial v}{\partial \nu} +\alpha v& = & 0 & \text{on} & \partial \Omega^{\#},
\end{array}
\]
where $0\leq f\in L^2(\Omega)$, $\alpha>0$, $\Omega^{\#}$ is a centered ball whose volume equals that of $\Omega$, and $\#$ denotes the Schwarz rearrangement. The authors show that $u$ and $v$ compare through their Lorentz norms; when the dimension $n=2$, it follows that
\[
\|u\|_{L^1(\Omega)}\leq \|v\|_{L^1(\Omega^{\#})} \qquad \textup{and} \qquad \|u\|_{L^2(\Omega)}\leq \|v\|_{L^2(\Omega^{\#})}.
\]
Also in dimension $n=2$, when $f=1$, the authors show that $u$ and $v$ compare as in \eqref{eq:uvcompare}. They therefore deduce a two-dimensional isoperimetric inequality for Robin torsional rigidity. See also the related paper \cite{ACNT}. But beyond these two papers, the author is not aware of any others that discuss pde comparison principles for Robin problems. In this paper, we prove several such results.

To state our results we require some notation. For $0<a<b<+\infty$, let
\[
A=A(a,b)=\{x\in \mathbb{R}^n:a<|x|<b\}
\]
denote a spherical shell with inner radius $a$ and outer radius $b$. For a function $f\in L^1(A)$, let $f^{\#}$ denote the cap symmetrization (for a precise definition, see Definitions \ref{def:decrearr}, \ref{def:spherrearr}, and \ref{def:capsymm}). Denote
\[
A^{\bigstar}=(a,b)\times (0,\pi)
\]
for the polar rectangle associated to $A$ and define functions $Jf,f^{\bigstar}:A^{\bigstar}\to \mathbb{R}$ a.e. by
\begin{align*}
Jf(r,\theta)&=\int_{K(\theta)}f(r\xi)\,d \sigma(\xi),\\
f^{\bigstar}(r,\theta)&=\int_{K(\theta)}f^{\#}(r\xi)\,d\sigma(\xi),
\end{align*}
where $\sigma$ denotes standard surface measure on $\mathbb{S}^{n-1}$ and
\[
K(\theta)=\{\xi \in \mathbb{S}^{n-1}:d(e_1,\xi)<\theta \}
\]
denotes the open polar cap on $\mathbb{S}^{n-1}$ centered at $e_1$ with radius $\theta$, computed in the spherical distance. Our first main result is a Robin comparison principle for the cap symmetrization in shells. Notice that our result allows for the possibility of mixed Robin/Neumann boundary conditions.

\begin{theorem}[Robin Comparison Principle on Shells]\label{Th:shell}
Say $0<a<b<+\infty$ and $A=A(a,b)\subseteq \mathbb{R}^n$ is a spherical shell with inner radius $a$ and outer radius $b$. Let $\alpha,\beta \geq 0$ and $f\in L^2(A)$; when $\alpha=\beta =0$, we also assume $\int_{A}f\,dx=0$. Say $u$ and $v$ solve the Poisson problems
\[
\begin{array}{rclccccrclcc}
-\Delta u & = & f & \text{in} & A, &  &  & -\Delta v & = & f^{\#} & \text{in} & A,\\
\frac{\partial u}{\partial \nu} +\alpha u & = & 0 & \text{on} & \{|x|=a\}, &  &  & \frac{\partial v}{\partial \nu} + \alpha v & = & 0 & \text{on} & \{|x|=a\},\\
\frac{\partial u}{\partial \nu} +\beta  u & = & 0 & \text{on} & \{|x|=b\}, &  &  & \frac{\partial v}{\partial \nu} + \beta  v & = & 0 & \text{on} & \{|x|=b\},
\end{array}
\]
where $f^{\#}$ denotes the cap symmetrization of $f$;  when $\alpha=\beta =0$, we also assume that $u$ and $v$ are normalized so that $\int_{A}u\,dx=\int_{A}v\,dx=0$. Then for almost every $r\in(a,b)$, the inequality
\[
u^{\bigstar}(r,\theta)\leq Jv(r,\theta)
\]
holds for each $\theta \in (0,\pi)$. In particular, for almost every $r\in(a,b)$ and for each convex function $\phi:\mathbb{R}\to \mathbb{R}$, we have
\[
\int_{\{|x|=r\}}\phi(u)\,dS \leq \int_{\{|x|=r\}}\phi(v)\,dS.
\]
\end{theorem}

By sending $\alpha$ (or $\beta$) to $+\infty$ in Theorem \ref{Th:shell}, the Robin boundary condition along $\{|x|=a\}$ (or $\{|x|=b\}$) turns into a Dirichlet boundary condition. After proving Theorem \ref{Th:shell}, we use this observation to deduce analogous comparison principles on shells with mixed Robin/Dirichlet boundary conditions. See Theorem \ref{Th:shellmixed}.

We also prove comparison principles on balls that requires a bit more notation. For $0<b<+\infty$, let 
\[
B=B(0,b)=\{x\in \mathbb{R}^n:|x|<b\}.
\]
Denote
\[
B^{\bigstar}=(0,b)\times (0,\pi)
\]
for the polar rectangle associated to $B$ and define functions $Jf,f^{\bigstar}:B^{\bigstar}\to \mathbb{R}$ a.e. by
\begin{align*}
Jf(r,\theta)&=\int_{K(\theta)}f(r\xi)\,d \sigma(\xi),\\
f^{\bigstar}(r,\theta)&=\int_{K(\theta)}f^{\#}(r\xi)\,d\sigma(\xi),
\end{align*}
with $f^{\#}$ the cap symmetrization of $f$ (see also the discussion preceding the proof of Theorem \ref{Th:ball} in Section 3). Our second main result is:

\begin{theorem}[Robin Comparison Principle on Balls]\label{Th:ball}
Say $0<b<+\infty$ and let $B\subseteq \mathbb{R}^n$ denote an open ball centered at $0$ with radius $b$. Let $\beta \geq 0$ and $f\in L^2(A)$; when $\beta=0$, we also assume $\int_{B}f\,dx=0$. Say $u$ and $v$ solve the Poisson problems
\[
\begin{array}{rclccccrclcc}
-\Delta u & = & f & \text{in} & B, &  &  & -\Delta v & = & f^{\#} & \text{in} & B,\\
\frac{\partial u}{\partial \nu} +\beta u & = & 0 & \text{on} & \{|x|=b\}, &  &  & \frac{\partial v}{\partial \nu} + \beta v & = & 0 & \text{on} & \{|x|=b\},
\end{array}
\]
where $f^{\#}$ denotes the cap symmetrization of $f$;  when $\beta=0$, we also assume that $u$ and $v$ are normalized so that $\int_{B}u\,dx=\int_{B}v\,dx=0$. Then for almost every $r\in(a,b)$, the inequality
\[
u^{\bigstar}(r,\theta)\leq Jv(r,\theta)
\]
holds for each $\theta \in (0,\pi)$. In particular, for almost every $r\in(a,b)$ and for each convex function $\phi:\mathbb{R}\to \mathbb{R}$, we have
\[
\int_{\{|x|=r\}}\phi(u)\,dS \leq \int_{\{|x|=r\}}\phi(v)\,dS.
\]
\end{theorem}

To state our final main result, we also introduce some additional notation. Let $D\subseteq \mathbb{R}^{n-1}$ denote a bounded $C^{\infty}$ domain and say $\ell>0$. Define a generalized cylinder by
\[
\Omega=D\times (0,\ell).
\]
Points of $\Omega$ are denoted by pairs $(x,y)$ with $x\in D$ and $y\in (0,\ell)$. Given a function $f\in L^1(\Omega)$, we can hold almost every $x\in D$ fixed, and take a decreasing rearrangement in the variable $y$. Doing so yields the monotone decreasing rearrangement of $f$ in the $y$-direction (for the precise definition, see Definitions \ref{def:decrearr} and \ref{def:decy}). Define functions $Jf,f^{\bigstar}:\Omega \to \mathbb{R}$ a.e. by
\begin{align*}
Jf(x,y)&=\int_0^yf(x,t)\,dt,\\
f^{\bigstar}(x,y)&=\int_0^yf^{\#}(x,t)\,dt,
\end{align*}
where $f^{\#}$ is the monotone decreasing rearrangement of $f$ in the $y$-direction. We are now prepared to state our last result, a comparison principle on cylinders with mixed boundary conditions.

\begin{theorem}[Mixed Robin/Neumann Comparison Principle on Cylinders]\label{th:cyl}
Let $\Omega=D\times (0,\ell)$ denote a cylinder whose base $D \subset \mathbb{R}^{n-1}$ is a bounded $C^{\infty}$ domain and say $f\in L^2(\Omega)$.  We divide the boundary $\partial \Omega$ into two pieces, $\partial \Omega_1=(D\times \{0\})\cup (D\times \{\ell\})$ and $\partial \Omega_2=\partial D \times [0,\ell]$. For $\alpha \geq 0$, let $u$ and $v$ be solutions to the Poisson problems
\[
\begin{array}{rclccccrclcc}
-\Delta u & = & f & \text{in} & \Omega, &  &  & -\Delta v & = & f^{\#} & \text{in} & \Omega,\\
\frac{\partial u}{\partial \nu} & = & 0 & \text{on} & \partial \Omega_1 &  &  & \frac{\partial v}{\partial \nu} & = & 0 & \text{on} & \partial \Omega_1,\\
\frac{\partial u}{\partial \nu} +\alpha u & = & 0 & \text{on} & \partial \Omega_2, &  &  & \frac{\partial v}{\partial \nu} + \alpha v & = & 0 & \text{on} & \partial \Omega_2,
\end{array}
\]
where $f^{\#}$ denotes the decreasing rearrangement of $f$ in the $y$-direction; when $\alpha=0$, we assume $f$, $u$, and $v$ have zero mean on $\Omega$. Assume that $f$ and $f^{\#}$ are continuous on $\overline{\Omega}$ and are sufficiently regular to ensure that the solutions $u,v$ belong to $C^2(\overline{\Omega})$. Then for each $x\in D$
\[
u^{\bigstar}(x,y)\leq Jv(x,y)
\]
for every $0\leq y\leq \ell$. In particular,
\[
\int_0^{\ell}\phi(u(x,y))\,d y\leq \int_0^{\ell}\phi(v(x,y))\,d y
\]
for each convex function $\phi:\mathbb{R}\rightarrow \mathbb{R}$.
\end{theorem} 

Observe that this result has stronger assumptions on the functions $f,f^{\#},u$, and $v$ than our previous main results. This is because the domain $\Omega$ is Lipschitz rather than $C^{\infty}$ smooth. In this Lipschitz setting, the author is not aware of any regularity results that guarantee smoothness of solutions when the forcing function $f$ has sufficient regularity. Our goal here is not to establish such a regularity result, but rather to illustrate that the techniques used to prove Theorems \ref{Th:shell} and \ref{Th:ball} adapt to other settings.

We note that in \cite{Periodic Steiner}, Brock proves a result related to Theorem \ref{th:cyl} using Green's functions. Again, his conclusions are weaker than those obtained here, as it appears his main goal is to compare the extreme values of $u$ and $v$ rather than their convex means or $L^p$-norms. Indeed, all of our comparison principles have corollaries comparing the $L^p$-norms of $u$ and $v$ (see Corollaries \ref{cor:Lpnorms}, \ref{cor:Lpball}, and \ref{cor:Lpcyl}).

The techniques used in this paper are similar to those used by Baernstein in \cite{Baernstein Cortona Volume}, the author in \cite{Langford1} and \cite{Langford3} to prove Neumann versions of Theorems \ref{Th:shell}, \ref{Th:ball}, and \ref{th:cyl}, and Baernstein, Laugesen, and the author in Chapter 10 of \cite{Baernstein Symmetrization in Analysis}. Indeed, the techniques employed here are drastically different from those used by Talenti to prove his original comparison result (and also those used in \cite{ANT}). Whereas Talenti's work relies on the classic isoperimetric inequality, the coarea formula, and a detailed analysis of a function's (in particular $u$'s) level sets, our work relies on $\bigstar$-functions, so-called ``subharmonicity'' results (certain weak differential inequalities), and a weak Maximum Principle. In a certain sense, the machinery used here is robust, as it can be used to prove Dirichlet, Neumann, and now Robin comparison principles.

The remainder of the paper is outlined as follows. Section 2 is divided into three subsections. In the first, we introduce star functions on general measure spaces; in the second subsection, we lay out all of the background necessary to prove Theorems \ref{Th:shell} and \ref{Th:ball}, including relevant rearrangements and subharmonicity results; in the third subsection, we present background for Theorem \ref{th:cyl}, introducing relevant rearrangements and subharmonicity results. In Section 3, we prove our main results and related consequences.

\section{Background}
\subsection{Robin Problems.}\label{Sub:robinfund} We start by collecting basic facts about solutions to Poisson problems with Robin boundary conditions. Say $\Omega \subseteq \mathbb{R}^n$ is a bounded Lipschitz domain, $\alpha \in L^{\infty}(\partial \Omega)$ is nonnegative and not identically zero, and $f\in L^2(\Omega)$. We consider the Robin Poisson problem
\begin{eqnarray}
-\Delta u & = & f\quad\text{in}\quad\Omega,\label{eq:RPpde}\\
\frac{\partial u}{\partial \nu} +\alpha u& = & 0\quad\text{on}\quad\partial\Omega. \label{eq:RPbc}
\end{eqnarray}
A function $u\in H^1(\Omega)$ is said to be a weak solution provided
\[
\int_{\Omega}\nabla u \cdot \nabla v\,dx+\int_{\partial \Omega}\alpha uv\,dS=\int_{\Omega}fv\,dx,\qquad v\in H^1(\Omega),
\]
where the boundary integral is interpreted in the trace sense. 

The bilinear form associated to the Robin problem is
\[
B[v,w]=\int_{\Omega}\nabla v \cdot \nabla w\,dx+\int_{\partial \Omega}\alpha vw\,dS,\qquad v,w\in H^1(\Omega).
\]
The boundedness of $\alpha$ on $\partial \Omega$ and the Trace Theorem give that $B$ is bounded,  i.e. there exists a positive constant $c_1$ where
\begin{equation*}\label{eq:Bbd}
|B[v,w]|\leq c_1\|v\|_{H^1(\Omega)}\|w\|_{H^1(\Omega)},\qquad v,w\in H^1(\Omega).
\end{equation*}
The theorems of Banach-Alaoglu and Rellich-Kondrachov, and the assumption that $\alpha$ is positive on a set of positive measure on $\partial \Omega$ give that $B$ is coercive. So, there exists a positive constant $c_2$ with
\begin{equation*}
B[v,v]\geq c_2\|v\|_{H^1(\Omega)}^2,\qquad v\in H^1(\Omega).
\end{equation*}
The Lax-Milgram Theorem now implies that for each $f\in L^2(\Omega)$, there exists a unique $u\in H^1(\Omega)$ where
\[
B[u,v]=\int_{\Omega}fv\,dx, \qquad v\in H^1(\Omega).
\]
In other words, the Robin problem \eqref{eq:RPpde}, \eqref{eq:RPbc} with data $f$ has a unique solution $u$. Moreover, by coercivity, the solution depends continuously on the data:
\begin{equation}\label{ineq:contdata}
\|u\|_{H^1(\Omega)}\leq \frac{1}{c_2}\|f\|_{L^2(\Omega)}.
\end{equation}

\subsection{Star Functions on General Measure Spaces.}
The goal of our paper is to compare the solutions of two Robin problems as in $\eqref{eq:RPpde},\eqref{eq:RPbc}$, one with initial data, and one where the data has been rearranged, or symmetrized. As we shall see, the two solutions compare via their star functions. So in this subsection, we collect basic facts about star functions on general measure spaces to be used throughout the paper. Much of the background presented in this subsection also appears in Section 2.1 of \cite{Langford1}. For more on applications of the star function, we direct the reader to the work of Baernstein \cite{Baernstein Edrei's Spread Conjecture}, \cite{Baernstein Integral means}, \cite{Baernstein how the star function}, \cite{Barenstein Star Function in Complex Analysis}.

Unless specified otherwise, $(X,\mu)$ denotes a finite measure space. We start by precisely defining what it means for two functions to have the same size. 
\begin{definition}
[Rearrangements] Suppose that $f\in L^{1}(X)$ and $g\in L^{1}(Y)$
are defined on finite measure spaces $(X,\mu)$ and $(Y,\nu)$. We say $f$ and
$g$ are rearrangements of each other provided
\[
\mu\left(\{x\in X:t<f(x)\}\right)  =  \nu\left(\{y\in Y:t<g(y)\}\right),\qquad t\in \mathbb{R}.
\]
\end{definition}
For a comprehensive treatment of rearrangements, we direct the interested reader to the recently published manuscript of Baernstein \cite{Baernstein Symmetrization in Analysis}. See also \cite{Kawohl1} and \cite{Leib and Loss Analysis}.

Given an integrable function, we can always build a decreasing function on an interval that rearranges the given function's values. This function is called the decreasing rearrangement and it plays a central role in defining other rearrangements (symmetrizations) throughout our paper.
\begin{definition}
[Decreasing Rearrangement]\label{def:decrearr}Let $f\in L^{1}(X)$. Define $f^{\ast}:[0,\mu(X)]\rightarrow[-\infty,+\infty]$
via
\[
f^{\ast}(t)=\begin{cases}
\underset{X}{\textup{ess\ sup}}\ f & \textup{if}\ t=0,\\
\inf\{s:\mu(\{x:s<f(x)\})\leq t\} & \textup{if}\ t\in(0,\mu(X)),\\
\underset{X}{\textup{ess\ inf}}\ f & \textup{if}\ t=\mu(X).
\end{cases}
\]
We call $f^{\ast}$ the decreasing rearrangement of $f$.
\end{definition}

We are now prepared to define star functions on general measure spaces.
\begin{definition}
[Star Function for a General Measure Space]\label{Def: Star Function on General Measure Space}Let
$f\in L^{1}(X)$. We define the star function of $f$ on the interval $[0,\mu(X)]$ by the formula
\[
f^{\bigstar}(t) =  \underset{\mu(E)=t}{\sup}\ {\displaystyle \int_{E}f\,d\mu},
\]
where the $\sup$ is taken over all measurable subsets $E\subseteq X$ with $\mu(E)=t$.
\end{definition}
Our next proposition collects two important facts about the star function. First, on nonatomic measure spaces, the $\sup$ defining $f^{\bigstar}$ is always achieved by some measurable subset $E$.  Second, there is a simple connection between the star function and the decreasing rearrangement. The result below appears as Proposition 9.2 of \cite{Baernstein Symmetrization in Analysis}.
\begin{proposition}\label{prop:genstarprop}
\label{prop:Star function achieved}Assume $f\in L^{1}(X)$ with $(X,\mu)$
a finite nonatomic measure space. Then for each $t\in[0,\mu(X)]$, there
exists a measurable subset $E\subseteq X$ with $\mu(E)=t$ such that
\[
f^{\bigstar}(t) = \int_{E}f\,d\mu.
\]
Moreover,
\[
f^{\bigstar}(t)=\int_0^tf^{\ast}(s)\,ds,
\]
where $f^{\ast}$ is the decreasing rearrangement of $f$.
\end{proposition}
The next result allows us to interpret star function inequalities in terms of convex means. See Propositions 10.1 and 10.3 of \cite{Baernstein Symmetrization in Analysis}.

\begin{proposition}
[Majorization]\label{Prop:Majorization}
Assume $(X,\mu)$ is a finite measure space and $u,v\in L^{1}(X)$. Then
\[
u^{\bigstar}  \leq  v^{\bigstar}
\]
on $[0,\mu(X)]$ if and only if
\[
\int_{X}\phi(u)\,d\mu  \leq  \int_{X}\phi(v)\,d\mu
\]
for every increasing convex function $\phi:\mathbb{R}\rightarrow\mathbb{R}$. If $\int_{X}u\,d\mu=\int_{X}v\, d\mu$, then the word ``increasing'' may be removed from the previous statement.
\end{proposition}

The final result in this subsection gives further consequences of star function inequalities when the functions of interest have the same mean. The following result appears as Proposition 10.3 of \cite{Baernstein Symmetrization in Analysis}.
\begin{corollary}\label{cor:Lpnormsgen}
\label{cor:consequences of u star leq v star} Say $u,v\in L^{1}(X)$ where $X\subseteq \mathbb{R}$ is measurable subset of finite measure and assume $u^{\bigstar}\leq v^{\bigstar}$ on $[0,\mu(X)]$. If $\int_{X}u\,d\mu=\int_{X}v\,d\mu$, then
\[
\|u\|_{L^{p}(X)}  \leq  \|v\|_{L^{p}(X)},\quad1\leq p\leq +\infty.
\]
Moreover,
\begin{align*}
\underset{X}{\esssup}\ u & \leq  \underset{X}{\esssup}\ v,\\
\underset{X}{\essinf}\ u & \geq  \underset{X}{\essinf}\ v,\\
\underset{X}{\osc}\ u & \leq  \underset{X}{\osc}\ v,
\end{align*}
where $\osc=\esssup - \essinf$.
\end{corollary}

\subsection{Spherical Shells: Symmetrization, Star Functions, and Subharmonicity.}
In this subsection, we collect several key definitions and results for functions defined on spherical shells. These results drive the proof of Theorem \ref{Th:shell}. Much of the information in this subsection also appears in Section 2.2 of \cite{Langford1}.

We first discuss rearrangements of functions defined on spheres. Denote
\[
\mathbb{S}^{n}=\{(\xi_{1},\xi_{2},\ldots,\xi_{n+1})\in\mathbb{R}^{n+1}:\xi_{1}^{2}+\xi_{2}^{2}+\cdots+\xi_{n+1}^{2}=1\}
\]
for the standard unit sphere. We write $\sigma$ for standard surface measure on $\mathbb{S}^n$ and $d$ for the spherical distance, computed using lengths of arcs of great circles. Let
\[
K(\theta) = \{\xi\in\mathbb{S}^{n}:d(\xi,e_{1})<\theta\}
\]
denote the open polar cap, or geodesic ball, on $\mathbb{S}^n$ centered at $e_{1}=(1,0,\ldots,0)$
with radius $\theta$.

We next define the spherical rearrangement. Notice that this definition is the analogue of the Schwarz rearrangement in the spherical setting.
\begin{definition}
[Spherical Rearrangement]\label{def:spherrearr}Given $F\in L^{1}(\mathbb{S}^{n})$, we
define $F^{\#}:\mathbb{S}^{n}\rightarrow[-\infty,+\infty]$ by the
formula
\begin{eqnarray*}
F^{\#}(\xi) & = & F^{\ast}\left(\sigma\left(K(\theta)\right)\right),
\end{eqnarray*}
where $\theta$ is the spherical distance between the point $\xi$
and $e_{1}$, and $F^{\ast}$ is the decreasing rearrangement of $F$.
We call $F^{\#}$ the spherical rearrangement of $F$.
\end{definition}

For $0< a<b<+\infty$, we write
\[
A=A(a,b)=\{x\in\mathbb{R}^{n}:a<|x|<b\}
\]
for the spherical shell with inner radius $a$ and outer radius $b$. If $f\in L^1(A)$, the cap symmetrization of $f$ is obtained by spherically rearranging $f$'s radial slice functions. More precisely, we have the following definition.
\begin{definition}
[Cap Symmetrization]\label{def:capsymm} Suppose $f\in L^{1}(A)$. We define $f^{\#}:A\rightarrow[-\infty,+\infty]$ as follows. By Fubini's Theorem, the slice function $f^{r}:\mathbb{S}^{n}\rightarrow\mathbb{R}$ defined by $f^{r}(\xi)=f(r\xi)$ belongs to $L^{1}(\mathbb{S}^{n-1})$ for almost every $r\in(a,b)$.  For such $r$, the cap symmetrization of $f$ on $\{|x|=r\}$ is defined by 
\begin{eqnarray*}
f^{\#}(r\xi) & = & (f^{r})^{\#}(\xi),
\end{eqnarray*}
 where $(f^{r})^{\#}$ denotes the spherical rearrangement of the
slice function $f^{r}$. We leave $f^{\#}$ undefined on those spheres
$\{|x|=r\}$ for which $f^{r}\not\in L^{1}(\mathbb{S}^{n-1})$.
\end{definition}
Given a function defined on a spherical shell $A(a,b)$, its star function is defined in the polar rectangle
\[
A^{\bigstar}=\{(r,\theta)\in\mathbb{R}^{2}:a<r<b\text{\ and\ }0<\theta<\pi\}.
\]
We have the following definition.
\begin{definition}
[Star Functions on Spherical Shells]\label{Def: Star Functions in Shells}If
$f\in L^{1}(A)$, define $f^{\bigstar}:A^{\bigstar}\rightarrow\mathbb{R}$
a.e. by the formula
\[
f^{\bigstar}(r,\theta)=  \underset{\sigma(E)=\sigma(K(\theta))}{\max}\ {\displaystyle \int_{E}f(r\xi)\ \text{d}\sigma(\xi)}=  \int_{K(\theta)}f^{\#}(r\xi)\ \text{d}\sigma(\xi),
\]
where the $\max$ is taken over all measurable subsets $E$ of $\mathbb{S}^{n-1}$
with the same surface measure as $K(\theta)$ and $f^{\#}$ denotes
the cap symmetrization of $f$.
\end{definition}
Several remarks are in order. First, notice that the formula for $f^{\bigstar}(r,\theta)$ only makes sense if the radial slice function $f^{r}\in L^{1}(\mathbb{S}^{n-1})$. By Fubini's Theorem, $f^{\bigstar}(r,\theta)$ is well-defined for almost every $r\in (a,b)$ and for each $\theta \in (0,\pi)$. Second, the definition of $f^{\bigstar}$ given above is obtained by applying the star function of Definition \ref{Def: Star Function on General Measure Space} to each of $f$'s radial slice functions, up to a change of variable. Moreover, since $(\mathbb{S}^n,\sigma)$ is nonatomic,  Proposition \ref{prop:Star function achieved} tells us we are justified in using either $\sup$ or $\max$ to define $f^{\bigstar}$, since the $\max$ is obtained by some appropriate subset.

We close this subsection by stating ``commutativity'' and ``subharmonicity'' results
for cap symmetrization. With $A$ as above, suppose $u\in L^{1}(A)$. Define $Ju:A^{\bigstar}\rightarrow\mathbb{R}$
a.e. by
\begin{equation}\label{eq:Jucapshell}
Ju(r,\theta)=\int_{K(\theta)}u(r\xi)\,d\sigma(\xi).
\end{equation}
If $u^{\#}$ is the cap symmetrization of $u$, we therefore have
\begin{equation}\label{eq:ustareq}
u^{\bigstar}(r,\theta)=\int_{K(\theta)}u^{\#}(r\xi)\,d\sigma(\xi)=Ju^{\#}(r,\theta).
\end{equation}

Define operators $\Delta^{\bigstar}$ and $\Delta^{\bigstar t}$ acting on $F\in C^{2}(A^{\bigstar})$ via the formulas 
\begin{align}
\Delta^{\bigstar}F & =  \partial_{rr}F+\frac{n-1}{r}\partial_{r}F+r^{-2}[\partial_{\theta\theta}F-(n-2)(\cot\theta)\partial_{\theta}F],\label{eq:Delta star cap symmetrization-1}\\
\Delta^{\bigstar t}F & =  \partial_{rr}F-\frac{n-1}{r}\partial_{r}F+r^{-2}(n-1)F\nonumber  \nonumber \\
 & \qquad +r^{-2}[\partial_{\theta\theta}F+(n-2)(\cot\theta)\partial_{\theta}F-(n-2)(\csc^{2}\theta)F].\nonumber 
\end{align}
Then for $F\in C^2(A^{\bigstar})$ and $G\in C_c^2(A^{\bigstar})$, we have
\[
\int_{A^{\bigstar}}(\Delta ^{\bigstar}F)G\,dr\,d\theta=\int_{A^{\bigstar}}F\Delta^{\bigstar t}G\,dr\,d\theta.
\]
In other words, $\Delta^{\bigstar t}$ is formally the adjoint of $\Delta^{\bigstar}$ when $A^{\bigstar}$ is equipped with the measure $dr\,d\theta$.

The following results appear as Theorems 2.18 and 2.20 in \cite{Langford1}.
\begin{theorem}
[Commutativity Relation for Cap Symmetrization]\label{thm:J comm relation n dim}If
$u\in C^{2}(A)$, then 
\[
J\Delta u = \Delta^{\bigstar}Ju
\]
on $A^{\bigstar}$.
\end{theorem}

Our last result requires the following definition:
\begin{definition}
For $u\in C^{2}(A)$ and $f\in L_{loc}^{1}(A)$, we say that $-\Delta^{\bigstar}u^{\bigstar}\leq f^{\bigstar}$ holds
in the weak sense if
\[
-\int_{A^{\bigstar}}u^{\bigstar}\Delta^{\bigstar t}G\,dr\,d\theta\leq\int_{A^{\bigstar}}f^{\bigstar}G\,dr\,d\theta
\]
for every nonnegative $G\in C_{c}^{2}(A^{\bigstar})$.\end{definition}
We finally have the following result, the major tool in proving Theorem \ref{Th:shell}.
\begin{theorem}
[Subharmonicity for Cap Symmetrization]\label{thm:Subharmonicity n dim cap}Suppose
$u\in C^{2}(A)$ satisfies $-\Delta u=f$. Then
\[
-\Delta^{\bigstar}u^{\bigstar}  \leq  f^{\bigstar}
\]
 in the weak sense.
 \end{theorem}
 
\subsection{Generalized Cylinders: Symmetrization, Star Functions, and Subharmonicity.}
The structure of this subsection mimics the structure of the last one, as our goal is to gather all of the machinery necessary to prove our third main result, Theorem \ref{th:cyl}. Much of the information gathered here also appears in Sections 2 and 3 of \cite{Langford3}. The major change in this subsection is that here, we consider cylindrical domains rather than shell domains. To start, let $D\subseteq \mathbb{R}^{n-1}$ denote a bounded $C^{\infty}$ domain and say $\ell>0$. Define
\begin{equation}\label{eq:omegadef}
\Omega=D\times (0,\ell).
\end{equation}
Points in $\Omega$ will be denoted by pairs $(x,y)$ with $x\in D$ and $y\in (0,\ell)$. Given a function $f$ on $\Omega$, we can perform a decreasing rearrangement in the $y$-variable, holding each $x\in D$ fixed. Doing so gives the decreasing rearrangement of $f$ in the $y$-direction. More precisely, we have:
\begin{definition}[Decreasing Rearrangement in the $y$-direction]\label{def:decy}
Let $f\in L^1(\Omega)$ with $\Omega$ as in  \eqref{eq:omegadef}. Define a function $f^{\#}:\Omega\to [-\infty,+\infty]$ a.e. by
\[
f^{\#}(x,y)=(f^x)^{\ast}(y),
\]
where $f^x$ denotes the slice function that maps $y\in (0,\ell)$ to $f(x,y)$ and $\ast$ denotes the decreasing rearrangement. We call $f^{\#}$  the \textit{decreasing rearrangement of $f$ in the y-direction}.
\end{definition}

Given $u \in L^1({\Omega})$, Fubini's Theorem guarantees that for almost every $x\in D$, the slice function $f^x\in L^1(0,\ell)$. We are therefore led to define $Ju$ and $u^{\bigstar} $ a.e. on $\Omega$ via
\[
Ju(x,y)=\int_0^yu(x,t)\,d t
\]
and
\begin{equation}\label{eq:ustarcyl}
u^{\bigstar}(x,y)=Ju^{\#}(x,y)=\int_0^yu^{\#}(x,t)\,d t
\end{equation}
where $u^{\#}$ denotes the monotone decreasing rearrangement of $u$ in the $y$-direction. By Proposition \ref{prop:genstarprop},
\[
u^{\bigstar}(x,y)=\underset{\mathcal L^1(E)=y}{\max}\int_Eu(x,t)\,d t,
\]
where the $\max$ is taken over all measurable subsets $E$ of $(0,\ell)$ with one-dimensional Lebesgue measure (length) equal to $y$.

We are now prepared to state our commutativity result for the cylindrical setting. See Proposition 2 of \cite{Langford2}.

\begin{proposition}[Commutativity for Decreasing Rearrangement in $y$-direction]\label{Prop:CommutativityCyl}
Let $\Omega$ be as in \eqref{eq:omegadef}. If $u\in C^2(\overline{\Omega})$ satisfies $u_y(x,0)=0$ for $x\in D$, then
\[
J\Delta u=\Delta Ju
\]
in $\Omega$.
\end{proposition}

We finally have the following subharmonicity result in the cylindrical setting. See Theorem 1 of \cite{Langford2}.

\begin{theorem}[Subharmonicity for Decreasing Rearrangement in $y$-direction]\label{Thm:SubharmonicityCyl}
Let $\Omega$ be as in \eqref{eq:omegadef}. Say $u\in C^2(\overline{\Omega})$ and $u_y(x,0)=u_y(x,\ell)=0$ for $x\in D$. If $-\Delta u=f$ in $\Omega$, then
\[
-\Delta u^{\bigstar}\leq f^{\bigstar}
\]
in the weak sense in $\Omega$. That is,
\[
-\int_{\Omega}u^{\bigstar}\Delta G \,dx \,dy \leq \int_{\Omega}f^{\bigstar}G \,d x\,dy
\]
for nonnegative $G\in C_c^2(\Omega$).
\end{theorem}

\section{Proofs of Main Results}

We start this section with a proof of our paper's first main result.
\begin{proof}[Proof of Theorem \ref{Th:shell}] The proof is long, and so we break it down into several steps.

{\bf Step 1: Maximum Principle.} First suppose that $f\in C_c^{\infty}(A)$ is a test function. In particular, $f$ is Lipschitz continuous on $\mathbb{R}^n$. Then $f^{\#}$ is also Lipschitz continuous on $\mathbb{R}^n$; see Section 7.5 of \cite{Baernstein Symmetrization in Analysis} or Lemma 6.7 in \cite{Sarvas}.  By standard regularity results, the solutions $u,v$ belong to $C^2(\overline{A})$. See Theorem 6.31 in \cite{GilbargTrudinger} and the discussion following its proof on Fredholm alternatives or Theorem 3.2 of \cite{Olga}. See also \cite{Nardi} for an explicit discussion of the Neumann problem.

Put
\[
w(r,\theta)=u^{\bigstar}(r,\theta)-Jv(r,\theta)-\varepsilon Q(r,\theta), 
\]
for $(r,\theta)\in A^{\bigstar}$, where
\[
Q(r,\theta)=(r-a)(r-b)+C\theta(\pi-\theta).
\]
A direct computation shows
\begin{equation*}
\Delta^{\bigstar}Q(r,\theta)=2+\frac{n-1}{r}(2r-a-b)-\frac{C}{r^{2}}\left(2+(n-2)(\pi-2\theta)\cot\theta\right).
\end{equation*}
Since $(\pi-2\theta)\cot\theta\geq 0$ on $(0,\pi)$, we may choose $C>0$ sufficiently large to guarantee $\Delta^{\bigstar}Q\leq 0$ on $A^{\bigstar}$. We then compute in the weak sense
\begin{align*}
-\Delta^{\bigstar} w &= -\Delta^{\bigstar}u^{\bigstar} +\Delta^{\bigstar}Jv + \varepsilon \Delta^{\bigstar}Q\\
&\leq f^{\bigstar} + J\Delta v + \varepsilon \Delta^{\bigstar}Q\\
&=f^{\bigstar} - f^{\bigstar} + \varepsilon \Delta^{\bigstar}Q\\
&\leq 0,
\end{align*}
where the first inequality follows from Theorems \ref{thm:J comm relation n dim} and \ref{thm:Subharmonicity n dim cap}, the second equality follows from \eqref{eq:ustareq}, and the second inequality follows from our choice of $C$. The Maximum Principle for weakly $\Delta^{\bigstar}$-subharmonic functions (see Theorem 1 of \cite{Littman}, for example) implies
\begin{equation}\label{eq:maxprinc}
\sup_{A^{\bigstar}}w \leq \max_{\partial A^{\bigstar}}w.
\end{equation}
If $\displaystyle \max_{\partial A^{\bigstar}}w \leq 0$, then \eqref{eq:maxprinc} implies
\[
\sup_{A^{\bigstar}}(u^{\bigstar}-Jv)\leq \varepsilon \|Q\|_{L^{\infty}(A^{\bigstar})},
\]
and sending $\varepsilon \to 0^+$ gives $u^{\bigstar}\leq Jv$ in $A^{\bigstar}$. In Steps 2 and 3 below, we therefore assume $\displaystyle \max_{\partial A^{\bigstar}}w > 0$.

{\bf Step 2: Boundary Analysis on Vertical Sides of $\partial A^{\bigstar}$}. We first claim that $\displaystyle \max_{\partial A^{\bigstar}}w$ is not achieved at a point where $r=a$ or $r=b$, i.e. on a vertical side of $\partial A^{\bigstar}$. For $\theta \in [0,\pi]$ and $h>0$ small, let $E(a+h,\theta) \subseteq \mathbb{S}^{n-1}$ denote a measurable subset with $\sigma(E(a+h,\theta))=\sigma (K(\theta))$ for which the $\sup$ defining $u^{\bigstar}(a+h,\theta)$ is achieved. We estimate
\begin{align*}
& \frac{e^{-\alpha(a+h)}w(a+h,\theta)-e^{-\alpha a}w(a,\theta)}{h} \\
&\qquad\qquad\qquad \qquad \qquad=e^{-\alpha(a+h)}\int_{E(a+h,\theta)}\frac{u((a+h)\xi)}{h}\,d\sigma(\xi)- e^{-\alpha a}\int_{E(a,\theta)}\frac{u(a\xi)}{h}\,d\sigma(\xi)\\
&\qquad\qquad\qquad \qquad \qquad \qquad -e^{-\alpha(a+h)}\int_{K(\theta)}\frac{v((a+h)\xi)}{h}\,d\sigma(\xi) +e^{-\alpha a}\int_{K(\theta)}\frac{v(a\xi)}{h}\,d\sigma(\xi) \\
&\qquad\qquad\qquad \qquad \qquad\qquad -\varepsilon \frac{e^{-\alpha(a+h)}Q(a+h,\theta)-e^{-\alpha a}Q(a,\theta)}{h}\\
&\qquad\qquad\qquad \qquad \qquad\geq \int_{E(a,\theta)}\frac{e^{-\alpha(a+h)}u((a+h)\xi)-e^{-\alpha a}u(a\xi)}{h}\,d\sigma(\xi)\\
&\qquad\qquad\qquad \qquad \qquad\qquad -\int_{K(\theta)}\frac{e^{-\alpha(a+h)}v((a+h)\xi)-e^{-\alpha a}v(a\xi)}{h}\,d\sigma(\xi)\\
&\qquad\qquad \qquad \qquad\qquad \qquad -\varepsilon \frac{e^{-\alpha(a+h)}Q(a+h,\theta)-e^{-\alpha a}Q(a,\theta)}{h}.\\.
\end{align*}
From the Dominated Convergence Theorem we deduce
\begin{align*}
\liminf_{h\to 0^+}\frac{e^{-\alpha(a+h)}w(a+h,\theta)-e^{-\alpha a}w(a,\theta)}{h} &\geq -e^{-\alpha a}\int_{E(a,\theta)}(-u_r(a\xi)+\alpha u(a\xi))\,d\sigma(\xi)\\
&\qquad +e^{-\alpha a}\int_{K(\theta)}(-v_r(a\xi)+\alpha v(a\xi))\,d\sigma(\xi)\\
&\qquad +\varepsilon\left(e^{-\alpha a}(b-a)+\alpha e^{-\alpha a}C\theta(\pi-\theta)\right)\\
&=\varepsilon\left(e^{-\alpha a}(b-a)+\alpha e^{-\alpha a}C\theta(\pi-\theta)\right),
\end{align*}
where the equality holds courtesy of the Robin boundary conditions. If $\displaystyle \max_{\partial A^{\bigstar}}w$ is achieved at $(a,\theta)$, it follows that for all $h>0$ sufficiently small
\begin{equation}\label{eq:wleftvert}
w(a+h,\theta)>e^{\alpha  h}w(a,\theta)+\frac{e^{\alpha h}h\varepsilon(b-a)}{2}>w(a,\theta),
\end{equation}
where the last inequality uses $w(a,\theta)>0$ and $\alpha \geq 0$. When $\theta\in (0,\pi)$, \eqref{eq:wleftvert} says that $\displaystyle \max_{\partial A^{\bigstar}}w$ is not achieved at a point with $r=a$, since otherwise the Maximum Principle \eqref{eq:maxprinc} would be violated. Likewise, when $\theta=0$ or $\theta=\pi$, $\displaystyle \max_{\partial A^{\bigstar}}w$ is not achieved at a point with $r=a$ since according to \eqref{eq:wleftvert}, $w$ assumes strictly larger values at other points $(r,\theta)\in\partial A^{\bigstar}$ with $r\in(a,b)$.

We next do a similar analysis of the behavior of $w$ at points along the boundary $\partial A^{\bigstar}$ with $r=b$. To this end, say $\theta \in [0,\pi]$ and $h>0$ is small. We estimate 
\begin{align*}
&\frac{e^{\beta (b-h)}w(b-h,\theta)-e^{\beta  b}w(b,\theta)}{h}\\
&\qquad\qquad\qquad \qquad \qquad= e^{\beta (b-h)}\int_{E(b-h,\theta)}\frac{u((b-h)\xi)}{h}\,d\sigma(\xi) - e^{\beta b}\int_{E(b,\theta)}\frac{u(b\xi)}{h}\,d\sigma(\xi)\\
&\qquad\qquad\qquad \qquad \qquad\quad -e^{\beta (b-h)}\int_{K(\theta)}\frac{v((b-h)\xi)}{h}\,d\sigma(\xi)+e^{\beta b}\int_{K(\theta)}\frac{v(b\xi)}{h}\,d\sigma(\xi)\\
&\qquad\qquad\qquad \qquad \qquad\quad -\varepsilon \frac{e^{\beta (b-h)}Q(b-h,\theta)-e^{\beta b}Q(b,\theta)}{h}\\
&\qquad\qquad\qquad \qquad \qquad\geq \int_{E(b,\theta)}\frac{e^{\beta (b-h)}u((b-h)\xi)-e^{\beta b}u(b\xi)}{h}\,d\sigma(\xi)\\
&\qquad\qquad\qquad \qquad \qquad\qquad -\int_{K(\theta)}\frac{e^{\beta (b-h)}v((b-h)\xi)-e^{\beta b}v(b\xi)}{h}\,d\sigma(\xi)\\
&\qquad\qquad\qquad \qquad \qquad\qquad -\varepsilon \frac{e^{\beta (b-h)}Q(b-h,\theta)-e^{\beta b}Q(b,\theta)}{h}.
\end{align*}
As before, we use the Robin boundary conditions to deduce
\begin{align*}
\liminf_{h\to 0^+}\frac{e^{\beta (b-h)}w(b-h,\theta)-e^{\beta  b}w(b,\theta)}{h} &\geq -e^{\beta b}\int_{E(b,\theta)}(u_r(b\xi)+\beta u(b\xi))\,d\sigma(\xi)\\
&\qquad +e^{\beta b}\int_{K(\theta)}(v_r(b\xi)+\beta v(b\xi))\,d\sigma(\xi)\\
&\qquad +\varepsilon\left(e^{\beta b}(b-a)+\beta e^{\beta b}C\theta(\pi-\theta)\right)\\
&=\varepsilon\left(e^{\beta b}(b-a)+\beta e^{\beta b}C\theta(\pi-\theta)\right).
\end{align*}
If $\displaystyle \max_{\partial A^{\bigstar}}w$ is achieved at $(b,\theta)$, then for all $h>0$ sufficiently small, we have 
\[
w(b-h,\theta)>e^{\beta  h}w(b,\theta)+\frac{e^{\beta h}h\varepsilon(b-a)}{2}>w(b,\theta),
\]
again as $w(b,\theta)>0$ and $\beta \geq 0$. By the same argument as before, the string of strict inequalities immediately above shows that $\displaystyle \max_{\partial A^{\bigstar}}w$ cannot be achieved at a point with $r=b$.

{\bf Step 3: Boundary Analysis on Horizontal Sides of $\partial A^{\bigstar}$.} We start by analyzing the behavior of $u^{\bigstar}(r,\pi)-Jv(r,\pi)$ in $r$. Define
\[
\psi(r)=u^{\bigstar}(r,\pi)-Jv(r,\pi)=\int_{\mathbb{S}^{n-1}}(u(r\xi)-v(r\xi))\,d\sigma(\xi), \quad r\in[a,b].
\]
For $a\leq r_1<r_2\leq b$, note that
\begin{align*}
r_2^{n-1}\psi'(r_2)-r_1^{n-1}\psi'(r_1)&= r_2^{n-1}\int_{\mathbb{S}^{n-1}}\left(u_r(r_2\xi)-v_r(r_2\xi)\right)\,d\sigma(\xi)\\
&\qquad -r_1^{n-1}\int_{\mathbb{S}^{n-1}}\left(u_r(r_1\xi)-v_r(r_1\xi)\right)\,d\sigma(\xi)\\
&=\int_{\partial A(r_1,r_2)}\left(\frac{\partial u}{\partial \nu}-\frac{\partial v}{\partial \nu}\right)\,dS\\
&=\int_{A(r_1,r_2)}\left(\Delta u - \Delta v \right)\,dx\\
&=\int_{A(r_1,r_2)}\left(-f+f^{\#}\right)\,dx\\
&=0,
\end{align*}
since $f$ and $f^{\#}$ are rearrangements of each other on each sphere $\{|x|=r\}$. It follows that
\begin{equation}\label{eq:k1}
r^{n-1}\int_{\mathbb{S}^{n-1}}(u_r(r\xi)-v_r(r\xi)\,d\sigma(\xi)=k_1
\end{equation}
is constant on $[a,b]$. Integrating this equation shows that there exists a constant $k_2$ where
\begin{equation}\label{eq:k2}
\int_{\mathbb{S}^{n-1}}(u(r\xi)-v(r\xi))\,d\sigma(\xi)=k_1\int_a^rt^{1-n}\,dt+k_2
\end{equation}
on the interval $[a,b]$. Using the Robin boundary conditions along $\{|x|=a\}$ and $\{|x|=b\}$ leads to the system of equations
\begin{align}
0&=\int_{\{|x|=a\}}\left(\left(\frac{\partial u}{\partial \nu}+\alpha u\right)-\left(\frac{\partial v}{\partial \nu}+\alpha v\right)\right)\,dS=-k_1+\alpha a^{n-1}k_2, \label{eq:firstk1}\\
0&=\int_{\{|x|=b\}}\left(\left(\frac{\partial u}{\partial \nu}+\beta u\right)-\left(\frac{\partial v}{\partial \nu}+\beta v\right)\right)\,dS=k_1+\beta b^{n-1}\left(k_1\int_a^bt^{1-n}\,dt+k_2\right).\nonumber
\end{align}
Combining these equations shows that
\[
0=k_2\left(\alpha a^{n-1}+\beta b^{n-1}+\alpha a^{n-1}\beta b^{n-1}\int_a^bt^{1-n}\,dt\right)
\]
If at least one of $\alpha,\beta$ is nonzero, the equation above implies that $k_2=0$ and \eqref{eq:firstk1} implies that $k_1=0$. Thus, \eqref{eq:k2} becomes
\begin{equation}\label{eq:uvsame}
\int_{\mathbb{S}^{n-1}}\left(u(r\xi)-v(r\xi)\right)\,d\sigma(\xi)=0,\quad r\in [a,b].
\end{equation}
If $\alpha =\beta =0$, choosing either $r=a$ or $r=b$ in \eqref{eq:k1} shows that $k_1=0$. Integrating \eqref{eq:k2} and using the normalization assumptions on $u,v$ gives
\[
0=\int_{A}(u-v)\,dx=k_2\int_a^br^{n-1}\,dr,
\]
so $k_2=0$ and equation \eqref{eq:uvsame} holds in this case as well. Since $u^{\bigstar}-Jv=0$ when $\theta=0$, we therefore deduce
\[
w(r,0)=w(r,\pi)=-\varepsilon(r-a)(r-b).
\]
Using \eqref{eq:maxprinc}, our work in Step 2 implies
\[
\sup_{A^{\bigstar}}(u^{\bigstar}-Jv)\leq \max_{r\in [a,b]}(-\varepsilon(r-a)(r-b))+ \varepsilon\|Q\|_{L^{\infty}(A^{\bigstar})}=\frac{\varepsilon}{4}(b-a)^2+\varepsilon \|Q\|_{L^{\infty}(A^{\bigstar})}.
\]
Sending $\varepsilon \to 0^+$ gives $u^{\bigstar}\leq Jv$ in $A^{\bigstar}$.

{\bf Step 4: Approximation for the General Result.} Now say $f\in L^2(A)$ is a general function. Let  $u,v$ be as in the statement of Theorem \ref{Th:shell}. Choose a sequence of test functions $f_k\in C^{\infty}_c(A)$ where $f_k\to f$ in $L^2(A)$. If $\alpha =\beta =0$, we assume $f,f_k$ have zero mean. Since cap symmetrization is a contraction in the $L^2$-distance (see Theorem 3 or Corollary 1 of \cite{Baernstein Cortona Volume}, for example), we also have $f_k^{\#} \to f^{\#}$ in $L^2(A)$. Let $u_k$ and $v_k$ be solutions to
\[
\begin{array}{rclccccrclcc}
-\Delta u_k & = & f_k & \text{in} & A, &  &  & -\Delta v_k & = & f_k^{\#} & \text{in} & A,\\
\frac{\partial u_k}{\partial \nu} +\alpha  u_k & = & 0 & \text{on} & \{|x|=a\}, &  &  & \frac{\partial v_k}{\partial \nu} + \alpha  v_k & = & 0 & \text{on} & \{|x|=a\},\\
\frac{\partial u_k}{\partial \nu} +\beta  u_k & = & 0 & \text{on} & \{|x|=b\}, &  &  & \frac{\partial v_k}{\partial \nu} + \beta  v_k & = & 0 & \text{on} & \{|x|=b\}.
\end{array}
\]
If $\alpha =\beta =0$, we also assume $u_k,v_k$ have zero mean. By our work in Steps 1-3, for each $k$,
\begin{equation}\label{ineq:uvkstar}
u_k^{\bigstar}(r,\theta)\leq Jv_k(r,\theta)
\end{equation}
for $(r,\theta)\in A^{\bigstar}$ and
\begin{equation}\label{eq:ukvksamemean}
\int_{\mathbb{S}^{n-1}}(u_k(r\xi)-v_k(r\xi))\,d\sigma(\xi)=0
\end{equation}
for $r\in (a,b)$. If either $\alpha $ or $\beta $ is nonzero, then \eqref{ineq:contdata} implies $u_k\to u$ in $L^2(A)$. If $\alpha =\beta =0$, we still have convergence in $L^2(A)$; see Corollary 1.29 of \cite{LangfordThesis}, for instance. In particular, we have convergence in $L^1(A)$.  Spelled out, we have
\[
\lim_{k\to +\infty}\int_{a}^{b}\int_{\mathbb{S}^{n-1}}\big|u_{k}(r\xi)-u(r\xi)\big|\,d\sigma(\xi)\ r^{n-1}\,dr=0.
\]
Thus, by passing to a subsequence, we may assume
\begin{equation}\label{eq:uktoulim}
\lim_{k\to +\infty}\int_{\mathbb{S}^{n-1}}\big|u_{k}(r\xi)-u(r\xi)\big|\,d\sigma(\xi)=0
\end{equation}
for almost every $r\in (a,b)$. Since cap symmetrization is a contraction in the $L^1$-distance (again, see Theorem 3 of \cite{Baernstein Cortona Volume}), this last equality implies that for almost every $r\in(a,b)$,
\begin{equation}\label{eq:limuk}
\lim_{k\to +\infty}u_k^{\bigstar}(r,\theta)=u^{\bigstar}(r,\theta)
\end{equation}
for each $\theta \in (0,\pi)$. By an identical argument, we may pass to another subsequence and assume that for almost every $r\in (a,b)$
\begin{equation}\label{eq:vktoulim}
\lim_{k\to +\infty}\int_{\mathbb{S}^{n-1}}\big|v_{k}(r\xi)-v(r\xi)\big|\,d\sigma(\xi)=0
\end{equation}
and
\begin{equation}\label{eq:limvk}
\lim_{k\to +\infty}Jv_k(r,\theta)=Jv(r,\theta)
\end{equation}
for each $\theta\in(0,\pi)$. Sending $k\to +\infty$ in \eqref{eq:ukvksamemean} and using \eqref{eq:uktoulim} and \eqref{eq:vktoulim} shows that
\begin{equation}\label{eq:uvsamemeanstep4}
\int_{\mathbb{S}^{n-1}}\left(u(r\xi)-v(r\xi)\right)\,d\sigma(\xi)=0
\end{equation}
for almost every $r\in(a,b)$. Finally, sending $k\to +\infty$ in \eqref{ineq:uvkstar} and using \eqref{eq:limuk} and \eqref{eq:limvk} gives the theorem's first conclusion. The theorem's conclusion about convex means follows from   \eqref{eq:uvsamemeanstep4} and Proposition \ref{Prop:Majorization}.

\end{proof}

As a consequence of Theorem \ref{Th:shell}, we see that the solution to the symmetrized problem is itself symmetrized.
\begin{corollary}\label{cor:vsymmshell}
The solution $v$ in Theorem \ref{Th:shell} satisfies $v=v^{\#}$ a.e.
\end{corollary}

\begin{proof}
First suppose that $f\in C_c^{\infty}(A)$ is a test function. Then as in the proof of Theorem \ref{Th:shell}, $f^{\#}$ is Lipschitz continuous and $v\in C^2(\overline{A})$. Letting $f^{\#}$ play the role of $f$, Theorem \ref{Th:shell} gives that $v^{\bigstar}(r,\theta)\leq Jv(r,\theta)$ on $A^{\bigstar}$. The reverse inequality holds by definition of the star function, and so $v^{\bigstar}=Jv$ on $A^{\bigstar}$. When the dimension $n=2$, we use complex notation to write this equality as
\[
\int_{-\theta}^{\theta}v(re^{i\phi})\,d\phi=\int_{-\theta}^{\theta}v^{\#}(re^{i\phi})\,d\phi,
\]
for each $r\in (a,b)$ and $\theta \in (0,\pi)$. Differentiating this equation with respect to $\theta$ gives
\begin{equation}\label{eq:2dvsymm}
v(re^{i\theta})+v(re^{-i\theta})=2v^{\#}(re^{i\theta}).
\end{equation}
We claim that $v(re^{i\theta})=v(re^{-i\theta})$ for $\theta \in (0,\pi)$. If this were not the case, then say $v(re^{i\theta})>v(re^{-i\theta})$ for some $\theta$. We can then find $\varepsilon>0$ sufficiently small so that
\[
\int_{-\theta}^{\theta}v(re^{i\phi})\,d\phi<\int_{-\theta+\varepsilon}^{\theta+\varepsilon}v(re^{i\phi})\,d\phi,
\]
and this contradicts the fact that $v^{\bigstar}(r,\theta)=\int_{-\theta}^{\theta}v(re^{i\phi})\,d\phi$. Having verified the claim, \eqref{eq:2dvsymm} gives $v(re^{i\theta})=v^{\#}(re^{i\theta})$.

If $n\geq 3$, we write out the equation $v^{\bigstar}=Jv$ on $A^{\bigstar}$ using spherical coordinates. We have
\[
\int_0^\theta \int_{\mathbb{S}^{n-2}}v(r\cos t, ry\sin t)\,d\sigma(y)\sin^{n-2} t\,dt=\int_0^\theta \int_{\mathbb{S}^{n-2}}v^{\#}(r\cos t, ry\sin t)\,d\sigma(y)\sin^{n-2} t\,dt
\]
for each $r\in (a,b)$ and $\theta \in (0,\pi)$. Differentiating with respect to $\theta$ and dividing by $\sin^{n-2}\theta$ gives
\begin{equation}\label{eq:vsymm3higher}
\int_{\mathbb{S}^{n-2}}v(r\cos \theta, ry\sin \theta)\,d\sigma(y)=\int_{\mathbb{S}^{n-2}}v^{\#}(r\cos \theta, ry\sin \theta)\,d\sigma(y)
\end{equation}
for each $r\in (a,b)$ and $\theta \in (0,\pi)$. Note that $v^{\#}(r\cos \theta, ry\sin \theta)$ is constant in $y\in \mathbb{S}^{n-2}$. We claim that $v(r\cos \theta, ry\sin \theta)$ is also constant for $y\in \mathbb{S}^{n-2}$. Say this is not the case, and $v(r\cos \theta, ry_1\sin \theta)>v(r\cos \theta, ry_2\sin \theta)$ for $y_1,y_2\in \mathbb{S}^{n-2}$. Then we can find positive numbers $\varepsilon_1,\varepsilon_2>0$ where
\begin{equation}\label{ineq:B1B2v}
\inf_{z_1\in B_1} v(rz_1)>\sup_{z_2\in B_2}v(rz_2);
\end{equation}
here $B_1= B((\cos \theta, y_1\sin \theta),\varepsilon_1)$ and $B_2= B((\cos \theta, y_2\sin \theta),\varepsilon_2)$, where $B((\cos \theta, y_1\sin \theta),\varepsilon_1)$ denotes the spherical cap (or geodesic ball) on $\mathbb{S}^{n-1}$ centered at $(\cos \theta, y_1\sin \theta)$ with radius $\varepsilon_1$, and similarly for $B_2$. We additionally assume
\[
\sigma(B_1 \setminus K(\theta))=\sigma(B_2 \cap K(\theta)).
\]

Write
\[
K_1=(K(\theta)\setminus B_2))\cup (B_1\setminus K(\theta)).
\]
Then $\sigma(K_1)=\sigma(K(\theta))$, and using \eqref{ineq:B1B2v}, we have
\begin{align*}
\int_{K_1}v(r\xi)\,d\sigma(\xi)&=\int_{K(\theta)\setminus B_2}v(r\xi)\,d\sigma(\xi)+\int_{B_1\setminus K(\theta)}v(r\xi)\,d\sigma(\xi)\\
&>\int_{K(\theta)\setminus B_2}v(r\xi)\,d\sigma(\xi)+\int_{B_2\cap K(\theta)}v(r\xi)\,d\sigma(\xi)\\
&=\int_{K(\theta)}v(r\xi)\,d\sigma(\xi),
\end{align*}
which contradicts $v^{\bigstar}=Jv$. We conclude that $v(r\cos \theta, ry\sin \theta)$ is constant for $y\in \mathbb{S}^{n-2}$. Equation \eqref{eq:vsymm3higher} implies
\[
v(r\cos \theta, ry\sin \theta)=v^{\#}(r\cos \theta, ry\sin \theta)
\]
for each $y\in \mathbb{S}^{n-2}$, $r\in (a,b)$, and $\theta \in (0,\pi)$. That is, $v=v^{\#}$.

If $f\in L^2(A)$ is a general function, pick a sequence of test functions $f_k\in C_c^{\infty}(A)$ where $f_k\to f$ in $L^2(A)$. If $\alpha =\beta =0$, we assume each $f_k$ has zero mean. As cap symmetrization is a contraction in the $L^2$-distance, $f_k^{\#} \to f^{\#}$ in $L^2(A)$. As before, each $f_k^{\#}$ is Lipschitz continuous on $\overline{A}$. Let $v_k$ solve
\[
\begin{array}{rclcc}
-\Delta v_k & = & f_k^{\#} & \text{in} & A,\\
\frac{\partial v_k}{\partial \nu} + \alpha  v_k & = & 0 & \text{on} & \{|x|=a\},\\
\frac{\partial v_k}{\partial \nu} + \beta  v_k & = & 0 & \text{on} & \{|x|=b\}.
\end{array}
\]
If $\alpha =\beta =0$, we assume each $v_k$ has zero mean. As in the proof of Theorem \ref{Th:shell}, $v_k\to v$ in $L^2(A)$. By the work above, $v_k=v_k^{\#}$ for each $k$. Again, cap symmetrization is a contraction in the $L^2$-distance, and so $v_k^{\#}\to v^{\#}$ in $L^2(A)$. We deduce that $v=v^{\#}$ a.e., as desired.
\end{proof}

Another, more immediate consequence of Theorem \ref{Th:shell} is that cap symmetrizing the data increases $L^p$-norms.

\begin{corollary}\label{cor:Lpnorms}
Let $u$ and $v$ be as in Theorem \ref{Th:shell}. Then for almost every $r\in(a,b)$, the radial slice functions $u(r\ \! \cdot)$ and $v(r\ \! \cdot)$ satisfy
\begin{align*}
\|u(r\ \! \cdot)\|_{L^{p}(\mathbb{S}^{n-1})}  & \leq  \|v(r\ \!\cdot)\|_{L^{p}(\mathbb{S}^{n-1})},\quad1\leq p\leq +\infty,\\
\underset{\mathbb{S}^{n-1}}{\esssup}\ u(r\ \!\cdot) & \leq  \underset{\mathbb{S}^{n-1}}{\esssup}\ v(r\ \!\cdot),\\
\underset{\mathbb{S}^{n-1}}{\essinf}\ u(r\ \!\cdot) & \geq  \underset{\mathbb{S}^{n-1}}{\essinf}\ v(r\ \!\cdot),\\
\underset{\mathbb{S}^{n-1}}{\osc}\ u(r\ \!\cdot) & \leq  \underset{\mathbb{S}^{n-1}}{\osc}\ v(r\ \!\cdot),
\end{align*}
where $\osc=\esssup - \essinf$. Consequently,
\begin{align*}
\|u\|_{L^{p}(A)}  &\leq  \|v\|_{L^{p}(A)},\quad1\leq p\leq +\infty,\\
\underset{A}{\esssup}\ u & \leq  \underset{A}{\esssup}\ v,\\
\underset{A}{\essinf}\ u & \geq  \underset{A}{\essinf}\ v,\\
\underset{A}{\osc}\ u & \leq  \underset{A}{\osc}\ v.
\end{align*}
\end{corollary}

\begin{proof}
The conclusions about slice functions follow from Theorem \ref{Th:shell}, equation \eqref{eq:uvsamemeanstep4}, and Corollary \ref{cor:Lpnormsgen}; the remaining conclusions clearly follow.
\end{proof}

As remarked in the Introduction, letting the Robin parameter $\alpha$ or $\beta$ tend to $+\infty$, Theorem \ref{Th:shell} yields comparison principles for Poisson problems with mixed Robin/Dirichlet boundary conditions. We now make this precise. With $A$ as in Theorem \ref{Th:shell}, denote
\[
H^1_a(A)=\left\{u\in H^1(A):u=0\textup{ on }\{|x|=a\}\right\},
\]
where the boundary condition is understood in the trace sense. Given $f\in L^2(A)$ and $\beta\geq 0$, a function $u\in H_a^1(A)$ is said to be a weak solution to the mixed Poisson problem
\begin{eqnarray*}
-\Delta u & = & f\quad\text{in}\quad A,\\
u& = & 0\quad\text{on}\quad \{|x|=a\},\\
\frac{\partial u}{\partial \nu} +\beta u& = & 0\quad\text{on}\quad \{|x|=b\},
\end{eqnarray*}
provided
\[
\int_{A}\nabla u \cdot \nabla v\,dx+\beta\int_{\{|x|=b\}}uv\,dS=\int_{A}fv\,dx
\]
for all $v\in H_a^1(A)$. As in Section \ref{Sub:robinfund}, the above mixed problem admits a unique weak solution for each such $f$ and $\beta$. Here now is our mixed comparison principle in shells.

\begin{theorem}[Mixed Robin/Dirichlet Comparison Principle on Shells]\label{Th:shellmixed}
Say $0<a<b<+\infty$ and $A=A(a,b)\subseteq \mathbb{R}^n$ is a spherical shell with inner radius $a$ and outer radius $b$. Let $\beta \geq 0$ and $f\in L^2(A)$. Say $u$ and $v$ solve the mixed Poisson problems
\[
\begin{array}{rclccccrclcc}
-\Delta u & = & f & \text{in} & A, &  &  & -\Delta v & = & f^{\#} & \text{in} & A,\\
u & = & 0 & \text{on} & \{|x|=a\}, &  &  & v & = & 0 & \text{on} & \{|x|=a\},\\
\frac{\partial u}{\partial \nu} +\beta  u & = & 0 & \text{on} & \{|x|=b\}, &  &  & \frac{\partial v}{\partial \nu} + \beta  v & = & 0 & \text{on} & \{|x|=b\},
\end{array}
\]
where $f^{\#}$ denotes the cap symmetrization of $f$. Then for almost every $r\in(a,b)$, the inequality
\[
u^{\bigstar}(r,\theta)\leq Jv(r,\theta)
\]
holds for each $\theta \in (0,\pi)$. In particular, for almost every $r\in(a,b)$ and for each convex function $\phi:\mathbb{R}\to \mathbb{R}$, we have
\[
\int_{\{|x|=r\}}\phi(u)\,dS \leq \int_{\{|x|=r\}}\phi(v)\,dS.
\]
\end{theorem}

Our proof of Theorem \ref{Th:shellmixed} is essentially based on the observation that solutions to Poisson's equation depend continuously on the Robin parameters (in an appropriate sense), even as those parameters tend towards $+\infty$.

\begin{proof}[Proof of Theorem \ref{Th:shellmixed}]
Let $\alpha_k$ denote a strictly increasing positive sequence with $\alpha_k\to +\infty$. Let $u_k$ and $v_k$ denote solutions to the Poisson problems
\[
\begin{array}{rclccccrclcc}
-\Delta u_k & = & f & \text{in} & A, &  &  & -\Delta v & = & f^{\#} & \text{in} & A,\\
\frac{\partial u_k}{\partial \nu} +\alpha_k  u_k & = & 0 & \text{on} & \{|x|=a\}, &  &  &  \frac{\partial v_k}{\partial \nu} + \alpha_k  v_k & = & 0 & \text{on} & \{|x|=a\},\\
\frac{\partial u_k}{\partial \nu} +\beta  u_k & = & 0 & \text{on} & \{|x|=b\}, &  &  & \frac{\partial v_k}{\partial \nu} + \beta  v_k & = & 0 & \text{on} & \{|x|=b\}.
\end{array}
\]
By Theorem \ref{Th:shell} (and its proof), for almost every $r\in(a,b)$
\begin{equation}\label{eq:ukvksamemeanmixed}
\int_{\mathbb{S}^{n-1}}(u_k(r\xi)-v_k(r\xi))\,d\sigma(\xi)=0
\end{equation}
for each $k$ and
\begin{equation}\label{eq:ukvkmixed}
u_k^{\bigstar}(r,\theta)\leq Jv_k(r,\theta)
\end{equation}
for each $\theta \in (0,\pi)$. We claim that the functions $u_k$ are bounded in $H^1(A)$. To back our claim, first note that
\begin{align*}
&\int_{A}|\nabla u_k|^2\,dx+\alpha_k\int_{\{|x|=a\}}u_k^2\,dS+\beta\int_{\{|x|=b\}}u_k^2\,dS\\
&  \qquad \qquad  \qquad \qquad \geq \int_{A}|\nabla u_k|^2\,dx+\alpha_1\int_{\{|x|=a\}}u_k^2\,dS+\beta\int_{\{|x|=b\}}u_k^2\,dS\\
& \qquad \qquad   \qquad \qquad  \geq c_2\|u_k\|_{H^1(A)}^2,
\end{align*}
where $c_2$ denotes coercivity constant (see Section \ref{Sub:robinfund}) of the bilinear form associated to the Robin problem with Robin parameter $\alpha_1$ on the inner sphere $\{|x|=a\}$ and Robin parameter $\beta$ on the outer sphere $\{|x|=b\}$. As
\begin{equation}\label{eq:mixedlimit}
\int_{A}|\nabla u_k|^2\,dx+\alpha_k\int_{\{|x|=a\}}u_k^2\,dS+\beta\int_{\{|x|=b\}}u_k^2\,dS=\int_{A}fu_k\,dx,
\end{equation}
it follows that
\[
\|u_k\|_{H^1(A)}\leq \frac{1}{c_2}\|f\|_{L^2(A)}
\]
for every $k$. The $u_k$ are thus bounded in $H^1(A)$. By Banach-Alaoglu and Rellich-Kondrachov, we may pass to a subsequence and assume the existence of $u'\in H^1(A)$ where $u_k\to u'$ weakly in $H^1(A)$ and $u_k\to u'$ in $L^2(A)$. By p.134 of \cite{EvansGariepy},
\[
\int_{\partial A}w^2\,dS\leq C\left(\int_{A}|\nabla w||w|\,dx+\int_{A}w^2\,dx\right)
\]
for $w\in H^1(A)$ and for some universal constant $C$. Thus, $u_k\to u'$ in $L^2(\partial A)$. Dividing \eqref{eq:mixedlimit} by $\alpha_k$ and sending $k\to +\infty$ shows $\int_{\{|x|=a\}}(u')^2\,dS=0$, i.e $u'\in H^1_a(A)$.

If $v\in H_a^1(A)$, then for each $k$,
\begin{equation*}
\int_{A}\nabla u_k\cdot \nabla v\,dx+\beta\int_{\{|x|=b\}}u_kv\,dS=\int_{A}fv\,dx.
\end{equation*}
Thus, letting $k\to +\infty$ in the above equation and using weak convergence, we conclude
\[
\int_{A}\nabla u'\cdot \nabla v\,dx+\beta\int_{\{|x|=b\}}u'v\,dS=\int_{A}fv\,dx.
\]
By uniqueness, $u'=u$, and so $u_k\to u$ in $L^2(A)$. The argument in Step 4 in the proof of Theorem \ref{Th:shell} shows that for almost every $r\in(a,b)$,
\begin{equation}\label{eq:uksamemeanmixed}
\int_{\mathbb{S}^{n-1}}u_k(r\xi)\,d\sigma(\xi) \to \int_{\mathbb{S}^{n-1}}u(r\xi)\,d\sigma(\xi)
\end{equation}
and
\begin{equation}\label{eq:ukconvmided}
u^{\bigstar}_k(r,\theta)\to u^{\bigstar}(r,\theta)
\end{equation}
for each $\theta \in (0,\pi)$. An analogous argument shows that we may assume $v_k\to v$ in $L^2(A)$, and again, the proof of Theorem \ref{Th:shell} shows that for almost every $r\in(a,b)$,
\begin{equation}\label{eq:vksamemeanmixed}
\int_{\mathbb{S}^{n-1}}v_k(r\xi)\,d\sigma(\xi) \to \int_{\mathbb{S}^{n-1}}v(r\xi)\,d\sigma(\xi)
\end{equation}
and
\begin{equation}\label{eq:vkconvmided}
Jv_k(r,\theta)\to Jv(r,\theta)
\end{equation}
for each $\theta \in (0,\pi)$. Combining \eqref{eq:ukvkmixed}, \eqref{eq:ukconvmided}, and \eqref{eq:vkconvmided} gives the theorem's first conclusion. Moreover, pairing \eqref{eq:ukvksamemeanmixed} with \eqref{eq:uksamemeanmixed} and \eqref{eq:vksamemeanmixed} shows that for almost every $r\in(a,b)$,
\[
\int_{\mathbb{S}^{n-1}}(u(r\xi)-v(r\xi))\,d\sigma(\xi)=0.
\]
The equation above and Proposition \ref{Prop:Majorization} give the theorem's conclusion about convex means.
\end{proof}

Of course, letting $\beta \to +\infty$ in Theorem \ref{Th:shell} (rather than $\alpha \to +\infty$ as we have just done), one deduces an analogue of Theorem \ref{Th:shellmixed} with Robin boundary conditions on the inner sphere $\{|x|=a\}$ and Dirichlet boundary conditions on the outer sphere $\{|x|=b\}$. Similarly, letting both $\alpha,\beta \to +\infty$ in Theorem \ref{Th:shell}, one deduces a comparison principle with Dirichlet boundary conditions on both spheres $\{|x|=a\}$ and $\{|x|=b\}$. The conclusions of Corollary \ref{cor:Lpnorms} also hold for the solutions $u$ and $v$ of Theorem \ref{Th:shellmixed} (and also for the two additional results just mentioned).

We next prove our paper's second main result. But before we do so, we briefly recall how cap symmetrization is defined for functions on balls, as the definition has one technicality that does not arise in the shell setting. In our result, recall that we take $B=\{x\in\mathbb{R}^{n}:|x|<b\}$ with $0<b<+ \infty$. The cap symmetrization of a function $f\in L^{1}(B)$ is defined as follows. When $0<r<b$ and the slice function $f^{r}:\mathbb{S}^{n-1}\rightarrow\mathbb{R}$ belongs to $L^{1}(\mathbb{S}^{n-1})$,
we define
\[
f^{\#}(r\xi)=(f^{r})^{\#}(\xi),
\]
where $(f^{r})^{\#}$ denotes the spherical rearrangement of the slice
function $f^{r}$. If $f^{r}\notin L^{1}(\mathbb{S}^{n-1})$, we leave $f^{\#}$ undefined on the sphere $\{|x|=r\}$. If $f(0)$ is defined, we also define $f^{\#}(0)=f(0)$. When $f(0)$ is not defined, $f^{\#}(0)$
is likewise left undefined. Note that the definitions of $Jf$ and $f^{\bigstar}$ in this setting (as defined in the Introduction), agree with Definition \ref{Def: Star Functions in Shells} and equation \eqref{eq:Jucapshell} with $a=0$. Here now is our second main result.

\begin{proof}[Proof of Theorem \ref{Th:ball}.] First say $f\in C_c^{\infty}(B)$ is a test function. As before, $u,v\in C^2(\overline{B})$. Say $\delta,\varepsilon>0$ and let $A_{\delta}=A(\delta,b)$ denote a spherical shell with inner radius $\delta$ and outer radius $b$. Let $w$ be defined as in the proof of Theorem \ref{Th:shell}:
\[
w(r,\theta)=u^{\bigstar}(r,\theta)-Jv(r,\theta)-\varepsilon Q(r,\theta),
\]
where
\[
Q(r,\theta)=r(r-b)+C\theta(\pi-\theta),
\]
and again $C$ is chosen so that
\begin{equation*}
\Delta^{\bigstar}Q(r,\theta)=2+\frac{n-1}{r}(2r-b)-\frac{C}{r^{2}}\left(2+(n-2)(\pi-2\theta)\cot\theta\right)\leq 0
\end{equation*}
on $B^{\bigstar}$. As before, we have
\[
-\Delta ^{\bigstar}w\leq 0
\]
weakly in $A_{\delta}^{\bigstar}$. By the Maximum Principle,
\begin{equation}\label{ineq:MPball}
\sup_{A_{\delta}^{\bigstar}}w \leq \max_{\partial A_{\delta}^{\bigstar}}w,
\end{equation}
and letting $\delta \to 0^+$, we see
\begin{equation}\label{ineq:T0}
\sup_{B^{\bigstar}}w \leq \max_{\partial B^{\bigstar}}w.
\end{equation}
As before, we may assume $\displaystyle \max_{\partial B^{\bigstar}}w>0$.

Defining
\[
\psi(r)=\int_{\mathbb{S}^{n-1}}(u(r\xi)-v(r\xi))\,d\sigma(\xi),\qquad r\in[0,b],
\]
the proof of Theorem \ref{Th:shell} shows that $r^{n-1}\psi'(r)=k_1$ is constant on $(0,b]$. Letting $r\to 0^+$ and using $u,v\in C^2(\overline{B})$, we conclude $k_1=0$. Thus
\begin{equation}\label{eq:balluvmeanconstant}
\int_{\mathbb{S}^{n-1}}(u(r\xi)-v(r\xi))\,d\sigma(\xi)=k_2
\end{equation}
is constant on $[0,b]$. We therefore have
\begin{align*}
b^{n-1}\int_{\mathbb{S}^{n-1}}\left(u_r(b\xi)-v_r(b\xi)\right)\,d\sigma(\xi)&=0,\\
b^{n-1}\beta  \int_{\mathbb{S}^{n-1}}\left(u(b\xi)-v(b\xi)\right)\,d\sigma(\xi)&=b^{n-1}\beta  k_2.\\
\end{align*}
Adding these equations and using the Robin boundary conditions shows
\[
0=\int_{\{|x|=b\}}\left(\left(\frac{\partial u}{\partial \nu}+\beta  u\right)-\left(\frac{\partial v}{\partial \nu}+\beta  v\right)\right)\,dS=b^{n-1}\beta  k_2.
\]
If $\beta  \neq 0$, we conclude $k_2=0$, and so
\begin{equation}\label{eq:uvsamemeanball}
\int_{\mathbb{S}^{n-1}}(u(r\xi)-v(r\xi))\,d\sigma(\xi)=0,\qquad r\in[0,b].
\end{equation}
If $\beta =0$, we integrate \eqref{eq:balluvmeanconstant} and use our normalization assumptions on $u,v$:
\[
0=\int_{B}(u-v)\,dx=k_2\int_0^br^{n-1}\,dr,
\]
so $k_2=0$ and again \eqref{eq:uvsamemeanball} holds.

Step 2 in the proof of Theorem \ref{Th:shell} shows that $\displaystyle \max_{\partial B^{\bigstar}}w$ cannot be achieved at a point with $r=b$. Now $u^{\bigstar}-Jv=0$ by definition when $\theta=0$ and \eqref{eq:uvsamemeanball} implies that $u^{\bigstar}-Jv=0$ when $\theta=\pi$. Moreover, \eqref{eq:uvsamemeanball} implies $u(0)=v(0)$, and so $u^{\bigstar}-Jv=0$ when $r=0$. Thus, \eqref{ineq:T0} implies
\[
\sup_{B^{\bigstar}}(u^{\bigstar}-Jv) \leq \max_{\partial B^{\bigstar}}(-\varepsilon Q)+\varepsilon\|Q\|_{L^{\infty}(B^{\bigstar})}.
\]
Sending $\varepsilon \to 0^+$ shows $u^{\bigstar}-Jv\leq 0$ in $B^{\bigstar}$. The approximation argument from the proof of Theorem \ref{Th:shell} applies almost verbatim to our new setting, and shows that for a given function $f\in L^2(B)$, for almost every $r\in(0,b)$, the inequality $u^{\bigstar}(r,\theta)\leq Jv(r,\theta)$ holds for each $\theta \in (0,\pi)$. Equation \eqref{eq:uvsamemeanball} and the approximation argument in the proof of Theorem \ref{Th:shell} also shows that for general $f\in L^2(B)$, 
\begin{equation}\label{eq:samemeansball}
\int_{\mathbb{S}^{n-1}}(u(r\xi)-v(r\xi))\,d\sigma(\xi)=0
\end{equation}
for almost every $r\in(0,b)$. Thus, the theorem's conclusion about convex means follows from Proposition \ref{Prop:Majorization}. 
\end{proof}

The proof of Corollary \ref{cor:vsymmshell} similarly carries over verbatim to give the following result.

\begin{corollary}
The solution $v$ in Theorem \ref{Th:ball} satisfies $v=v^{\#}$ a.e.
\end{corollary}

Equation \eqref{eq:samemeansball}, Theorem \ref{Th:ball}, and Corollary \ref{cor:Lpnormsgen} similarly give the following analogue of Corollary \ref{cor:Lpnorms} for the ball setting.

\begin{corollary}\label{cor:Lpball}
Let $u$ and $v$ be as in Theorem \ref{Th:ball}. Then for almost every $r\in (0,b)$, we have
\[
\|u(r\ \! \cdot)\|_{L^{p}(\mathbb{S}^{n-1})}  \leq  \|v(r\ \!\cdot)\|_{L^{p}(\mathbb{S}^{n-1})},\quad1\leq p\leq +\infty,
\]
and
\begin{align*}
\underset{\mathbb{S}^{n-1}}{\esssup}\ u(r\ \!\cdot) & \leq  \underset{\mathbb{S}^{n-1}}{\esssup}\ v(r\ \!\cdot),\\
\underset{\mathbb{S}^{n-1}}{\essinf}\ u(r\ \!\cdot) & \geq  \underset{\mathbb{S}^{n-1}}{\essinf}\ v(r\ \!\cdot),\\
\underset{\mathbb{S}^{n-1}}{\osc}\ u(r\ \!\cdot) & \leq  \underset{\mathbb{S}^{n-1}}{\osc}\ v(r\ \!\cdot).
\end{align*}
Consequently,
\[
\|u\|_{L^{p}(B)}  \leq  \|v\|_{L^{p}(B)},\quad1\leq p\leq +\infty,
\]
and
\begin{align*}
\underset{B}{\esssup}\ u & \leq  \underset{B}{\esssup}\ v,\\
\underset{B}{\essinf}\ u & \geq  \underset{B}{\essinf}\ v,\\
\underset{B}{\osc}\ u & \leq  \underset{B}{\osc}\ v.
\end{align*}
\end{corollary}

We next prove our paper's third main result.

\begin{proof}[Proof of Theorem \ref{th:cyl}]
Again, the proof is rather long, so we break it down into manageable steps.

\textbf{Step 1: Maximum Principle.} If $\alpha>0$, let $Q_1$ denote the solution to the Robin problem
\[
\begin{array}{rclcc}
\Delta Q_1 & = & 0 & \text{in} & D,\\
\frac{\partial Q_1}{\partial \nu} +\alpha Q_1 & = & 1 & \text{on} & \partial D.
\end{array}
\]
If $\alpha=0$, let $Q_1$ be the solution to the Neumann problem
\[
\begin{array}{rclcc}
\Delta Q_1 & = & \frac{\mathcal H^{n-2}(\partial D)}{\mathcal L^{n-1}(D)} & \text{in} & D,\\
\frac{\partial Q_1}{\partial \nu} & = & 1 & \text{on} & \partial D,
\end{array}
\]
where $Q_1$ is additively normalized to have zero  mean. Here we have written $\mathcal H^{n-2}$ for $(n-2)$-dimensional Hausdorff measure and $\mathcal L^{n-1}$ for $(n-1)$-dimensional Lebesgue measure. By Theorem 6.31 in \cite{GilbargTrudinger}, Theorem 3.2 in \cite{Olga}, and \cite{Nardi}, $Q_1$ belongs to $C^2(\overline{D})$. We are led to define a function $Q$ by
\[
Q(x,y)=Q_1(x)+Cy(\ell-y),
\]
for $(x,y)\in \Omega$, where $C>0$ is chosen large enough to guarantee
\[
\Delta Q=\Delta_x Q_1-2C\leq 0
\]
on $\Omega$. Define for $\varepsilon>0$,
\[
w(x,y)=u^{\bigstar}(x,y)-Jv(x,y)-\varepsilon Q(x,y)
\]
for $(x,y)\in \Omega$.

In the weak sense, we have
\begin{align*}
-\Delta w &= -\Delta u^{\bigstar} +\Delta Jv + \varepsilon \Delta Q\\
&\leq f^{\bigstar} + J\Delta v + \varepsilon \Delta Q\\
&=f^{\bigstar} - f^{\bigstar} + \varepsilon \Delta Q\\
&\leq 0,
\end{align*}
where the first inequality follows from Proposition \ref{Prop:CommutativityCyl} and Theorem \ref{Thm:SubharmonicityCyl}, the second equality follows from \eqref{eq:ustarcyl}, and the second inequality follows from our choice of $C$.
Thus $w$ is weakly subharmonic in $\Omega$ so by the Maximum Principle (Theorem 1 in \cite{Littman})
\begin{equation}\label{ineq:MPcyl}
\sup_{\Omega}w\leq \max_{\partial \Omega}w.
\end{equation}
As before, we assume $\displaystyle \max_{\partial \Omega}w>0$, since otherwise, we are done.

\textbf{Step 2: Boundary Analysis on $\partial \Omega_2$.} We first show that $\displaystyle \max_{\partial \Omega}w$ is not achieved at a point of $\partial \Omega_2$. Fix $(x,y)\in \partial \Omega_2$, so that $x\in \partial D$ and $y\in [0,\ell]$. Let $n=n(x)$ denote the inner normal vector at $x$ to $\partial D$. Let $E(x)$ denote a measurable subset of $(0,\ell)$ of length $y$ for which the $\sup$ defining $u^{\bigstar}(x,y)$ is achieved. By translating, we assume $0\in D$. Write $r=|x|$. We then have for $h>0$,
\begin{align*}
&\frac{e^{-\alpha(r+h)}w(x+hn,y)-e^{-\alpha r}w(x,y)}{h} \qquad \qquad \qquad \qquad \qquad \qquad \qquad \qquad \qquad   \\
&\qquad \qquad \qquad \qquad \qquad \qquad= e^{-\alpha(r+h)}\int_{E(x+hn)}\frac{u(x+hn,t)}{h}\,d t - e^{-\alpha r}\int_{E(x)}\frac{u(x,t)}{h}\,d t \\
&\qquad \qquad \qquad \qquad \qquad \qquad\qquad -e^{-\alpha(r+h)}\int_0^y\frac{v(x+hn,t)}{h}\,d t +e^{-\alpha r}\int_0^y\frac{v(x,t)}{h}\,d t\\
&\qquad \qquad \qquad \qquad \qquad \qquad\qquad -\varepsilon \frac{e^{-\alpha(r+h)}Q(x+hn,y)-e^{-\alpha r}Q(x,y)}{h}\\
&\qquad \qquad \qquad \qquad \qquad \qquad \geq \int_{E(x)}\frac{e^{-\alpha(r+h)}u(x+hn,t)-e^{-\alpha r}u(x,t)}{h}\,d t\\
&\qquad \qquad \qquad \qquad \qquad \qquad\qquad -\int_0^y\frac{e^{-\alpha(r+h)}v(x+hn,t)-e^{-\alpha r}v(x,t)}{h}\,d t\\
&\qquad \qquad \qquad \qquad \qquad \qquad\qquad -\varepsilon \frac{e^{-\alpha(r+h)}Q(x+hn,y)-e^{-\alpha r}Q(x,y)}{h}.
\end{align*}
Using the Dominated Convergence Theorem, we therefore have
\begin{align*}
\underset{h\rightarrow 0^+}{\textup{lim inf}}\frac{e^{-\alpha(r+h)}w(x+hn,y)-e^{-\alpha r}w(x,y)}{h}& \geq -e^{-\alpha r}\int_{E(x)}\left(\frac{\partial u}{\partial \nu}(x,t)+\alpha u(x,t)\right)\,d t \\
&\qquad +e^{-\alpha r}\int_0^y\left(\frac{\partial v}{\partial \nu}(x,t)+\alpha v(x,t)\right)\,d t \\
&\qquad +\varepsilon e^{-\alpha r}\left(1+\alpha Cy(\ell-y)\right)\\
& = \varepsilon e^{-\alpha r}\left(1+\alpha Cy(\ell-y)\right),
\end{align*}
where the last equality holds courtesy of the Robin boundary conditions on $\partial \Omega_2$. If $\displaystyle \max_{\partial \Omega}w$ is achieved at $(x,y)$, then for all $h>0$ sufficiently small we have
\begin{equation}\label{ineq:wcylsides}
w(x+hn,y)>e^{\alpha h}w(x,y)+\frac{e^{\alpha h}h\varepsilon}{2}>w(x,y),
\end{equation}
as $w(x,y)>0$ and $\alpha \geq 0$. When $y\in (0,\ell)$, inequality \eqref{ineq:wcylsides} says that $\displaystyle \max_{\partial \Omega}w$ is not achieved at a point $(x,y)\in \partial \Omega_2$, since otherwise the Maximum Principle \eqref{ineq:MPcyl} would be violated. Similarly, when $y=0$ or $y=\ell$, $\displaystyle \max_{\partial \Omega}w$ is not achieved at a point $(x,y)\in \partial \Omega_2$ since by \eqref{ineq:wcylsides}, $w$ assumes strictly larger values at points of $\partial \Omega_1$.

\textbf{Step 3: Boundary Analysis on $\partial \Omega_1$.} We start by defining 
\[
\psi(x)=\int_0^{\ell}(u(x,t)-v(x,t))\,d t,\qquad x\in D.
\]
Using the Dominated Convergence Theorem and our assumption that $u,v\in C^2(\overline \Omega)$, we see that $\psi$ is harmonic in $D$:
\begin{align*}
\Delta \psi(x)&=\int_0^{\ell}(\Delta_x u(x,t)-\Delta_x v(x,t))\,d t\\
&=\int_0^{\ell}(-u_{yy}(x,t)-f(x,t)+v_{yy}(x,t)+f^{\#}(x,t))\,d t\\
&= \int_0^{\ell}(-f(x,t)+f^{\#}(x,t))\,d t-u_y(x,\ell)+u_y(x,0)+v_y(x,\ell)-v_y(x,0)\\
&=0.
\end{align*}
This last equality holds by the Neumann boundary conditions of $u,v$ along $\partial \Omega_1$ and since $f$ and $f^{\#}$ are rearrangements of each other. 

Taking $x\in \partial D$, we next investigate the Robin boundary condition of $\psi$:
\[
\frac{\partial \psi}{\partial \nu}(x)+\alpha\psi(x)=\int_0^{\ell}\left(\left(\frac{\partial u}{\partial \nu}(x,t)+\alpha u(x,t)\right)-\left(\frac{\partial v}{\partial \nu}(x,t)+\alpha v(x,t)\right)\right)\,d t=0,
\]
since $u$ and $v$ themselves satisfy Robin boundary conditions on $\partial \Omega_2$. If $\alpha=0$, note that $\psi$ has zero mean by our normalization assumption on $u,v$:
\[
\int_D\psi(x)\,d x=\int_{\Omega}(u-v)\,d x\,dy=0.
\]
The function $\psi$ therefore solves the problem
\[
\begin{array}{rclcc}
\Delta \psi & = & 0 & \text{in} & D,\\
\frac{\partial \psi}{\partial \nu} +\alpha \psi & = & 0 & \text{on} & \partial D,
\end{array}
\]
and when $\alpha=0$, $\int_{D}\psi\,dx=0$. Integrating by parts, $\psi$ satisfies the equation
\[
\int_{D}|\nabla \psi|^2\,dx=-\alpha\int_{\partial D}\psi^2\,dS,
\]
so we deduce $\psi \equiv 0$ in $D$. Since $u^{\bigstar}-Jv=0$ when $y=0$, we conclude that
\[
w(x,0)=w(x,\ell)=-\varepsilon Q_1(x)
\]
for $x\in D$.

Our work from Steps 2 and 3 and the Maximum Principle \eqref{ineq:MPcyl} together give
\[
\sup_{\Omega} (u^{\bigstar}-Jv)\leq \varepsilon \|Q\|_{L^{\infty}(\Omega)}+\sup_{D}(-\varepsilon Q_1).
\]
Sending $\varepsilon \to 0^+$ gives $u^{\bigstar}-Jv\leq 0$ in $\Omega$. Note that since $\psi(x)\equiv 0$ on $D$, for each $x\in D$, the slice functions $u(x,\cdot)$ and $v(x,\cdot)$ have the same mean on $(0,\ell)$. Thus Proposition \ref{Prop:Majorization} gives the theorem's claims on the convex means of $u$ and $v$.
\end{proof}

As before, we easily deduce that the solution of the symmetrized problem is itself symmetrized.

\begin{corollary}
Let $v$ be as in Theorem \ref{th:cyl}. Then $v=v^{\#}$ on $\Omega$.
\end{corollary}

\begin{proof}
Letting $f=f^{\#}$, Theorem \ref{th:cyl} gives that $v^{\bigstar}\leq Jv$ on $\Omega$. Since the reverse inequality holds by definition, we have $v^{\bigstar}= Jv$ on $\Omega$. Spelled out, this means that for each $x\in D$,
\[
\int_0^yv^{\#}(x,t)\,dt=\int_0^yv(x,t)\,dt
\]
for each $y\in (0,\ell)$. Differentiating the above equation with respect to $y$ shows that $v(x,y)=v^{\#}(x,y)$ on $\Omega$, as desired.
\end{proof}

Finally, we have an analogue of Corollaries \ref{cor:Lpnorms} and \ref{cor:Lpball} for the cylindrical setting. 

\begin{corollary}\label{cor:Lpcyl}
Let $u$ and $v$ be as in Theorem \ref{th:cyl}. Then for each $x\in D$, the slice functions $u(x,\cdot)$ and $v(x,\cdot)$ satisfy
\[
\|u(x,\cdot)\|_{L^{p}(0,\ell)}  \leq  \|v(x,\cdot)\|_{L^{p}(0,\ell)},\quad1\leq p\leq +\infty.
\]
Moreover, for each $x\in D$, 
\begin{align*}
\underset{(0,\ell)}{\sup}\ u(x,\cdot) & \leq  \underset{(0,\ell)}{\sup}\ v(x,\cdot),\\
\underset{(0,\ell)}{\inf}\ u(x,\cdot) & \geq  \underset{(0,\ell)}{\inf}\ v(x,\cdot),\\
\underset{(0,\ell)}{\osc}\ u(x,\cdot) & \leq  \underset{(0,\ell)}{\osc}\ v(x,\cdot).
\end{align*}
Consequently,
\[
\|u\|_{L^{p}(\Omega)}  \leq  \|v\|_{L^{p}(\Omega)},\quad1\leq p\leq +\infty,
\]
and
\begin{align*}
\underset{\Omega}{\sup}\ u & \leq  \underset{\Omega}{\sup}\ v,\\
\underset{\Omega}{\inf}\ u & \geq  \underset{\Omega}{\inf}\ v,\\
\underset{\Omega}{\osc}\ u & \leq  \underset{\Omega}{\osc}\ v.
\end{align*}
\end{corollary}

\begin{proof}
The proof of Theorem \ref{th:cyl} shows that for each $x\in D$,
\[
\int_0^{\ell}u(x,t)\,dt=\int_0^{\ell}v(x,t)\,dt.
\]
This equality together with Theorem $\ref{th:cyl}$ and Corollary \ref{cor:Lpnormsgen} give the claimed inequalities on the slice functions of $u$ and $v$, and as before, the subsequent inequalities follow.
\end{proof}

\subsection*{Concluding Remarks} In surveying the results of the present paper, and also those of \cite{ACNT} and \cite{ANT}, the reader may be left with a lingering question: Are there comparison results with negative Robin parameters? The answer to this question is unclear. The techniques employed here clearly hinge on the Robin parameters being nonnegative.

In the transition from the positive to the negative regime, interesting things can happen. One particularly interesting example comes from spectral theory. Consider, for example, the the eigenvalue problem for the Robin Laplacian. Precisely, given a Lipschitz domain $\Omega \subseteq \mathbb{R}^n$, the eigenvalues of the Robin Laplacian
\begin{eqnarray*}
-\Delta u_k & = & \lambda_k(\Omega; \alpha) u_k \quad\text{in}\quad\Omega,\label{eq:RPpde}\\
\frac{\partial u}{\partial \nu} +\alpha u& = & 0\quad\text{on}\quad\partial\Omega, \label{eq:RPbc}
\end{eqnarray*}
are known to satisfy
\[
\lambda_1(\Omega; \alpha)<\lambda_2(\Omega; \alpha)\leq \lambda_3(\Omega; \alpha) \leq \cdots \to +\infty.
\]
For $\alpha>0$, the work of Bossel \cite{Bossel} and Daners \cite{Daners} gives an isoperimetric inequality for the lowest Robin eigenvalue:
\begin{equation}\label{eq:FKRobin}
\lambda_1(\Omega; \alpha) \geq \lambda_1(\Omega^{\#}; \alpha),
\end{equation}
where $\Omega^{\#}$ is the centered open ball in $\mathbb{R}^n$ with $|\Omega^{\#}|=|\Omega|$. Naturally, Bareket \cite{Bareket} conjectured that \eqref{eq:FKRobin} still holds for $\alpha<0$. But P. Freitas and D. Krej\v{c}i\v{r}\'{\i}k \cite{FK} showed that Bareket's conjecture is false in general for negative $\alpha$, though they do show that in dimension $n=2$, \eqref{eq:FKRobin} holds when the parameter $\alpha$ is sufficiently close to zero. Taken in sum, this work suggests that establishing pde comparison results with negative Robin parameters might be difficult and subtle. The author plans to investigate this question, but in a separate, future project.


\begin{thebibliography}{BBMP1}

\bibitem{ACNT}
A. Alvino, F. Chiacchio, C. Nitsch, and C. Trombetti, \emph{Sharp estimates for solutions to elliptic problems with mixed boundary conditions}, to appear in Journal de Math\'ematiques Pures et Appliqu\'ees. Retrievable from https://arxiv.org/abs/2009.04320.

\bibitem{AlvinoDiaz Lions and Trombetti}A. Alvino, J. I. Diaz,
P. L. Lions, and G. Trombetti, \emph{Elliptic equations and Steiner
symmetrization,} Comm. Pure Appl. Math. 49 (1996), no. 3, 217-236.

\bibitem{ANT}
A. Alvino, C. Nitsch, and C. Trombetti, \emph{A Talenti comparison result for solutions to elliptic problems with Robin boundary conditions}, preprint. Retrievable from https://arxiv.org/abs/1909.11950.

\bibitem{AFK}
P. R. S. Antunes, P. Freitas and D. Krej\v{c}i\v{r}\'{\i}k, \emph{Bounds and extremal domains for Robin eigenvalues with negative boundary parameter}, Adv. Calc. Var. 10 (2017), 357--379.

\bibitem{Bareket}
M. Bareket, \emph{On an isoperimetric inequality for the first eigenvalue of a boundary value problem},
SIAM J. Math. Anal. 8 (1977), 280--287.

\bibitem{Baernstein Edrei's Spread Conjecture}A. Baernstein
II, \emph{Proof of Edrei's spread conjecture}, Proc. London Math.
Soc. (3) 26 (1973), 418--434.

\bibitem{Baernstein Integral means}A. Baernstein II, \emph{Integral
means, univalent functions and circular symmetrization}, Acta Math.
133 (1974), 139--169.

\bibitem{Baernstein how the star function}A. Baernstein II,
\emph{How the \textasteriskcentered{}-function solves extremal problems},
Proceedings of the International Congress of Mathematicians (Helsinki,
1978), Acad. Sci. Fennica, Helsinki, 1980, 639--644.

\bibitem{Baernstein Cortona Volume}A. Baernstein II, \emph{A
unified approach to symmetrization}, Partial differential equations
of elliptic type (Cortona, 1992), 47\textendash{}91, Sympos. Math.,
XXXV, Cambridge Univ. Press, Cambridge, 1994.

\bibitem{Barenstein Star Function in Complex Analysis}A. Baernstein
II, \emph{The} \emph{\textasteriskcentered{}-function in complex analysis},
Handbook of Complex Analysis, Vol. I: Geometric Function Theory, edited
by R. Koehnau, Elsevier Science, 2002, 229--271.

\bibitem{Baernstein Symmetrization in Analysis} A. Baernstein II, ``Symmetrization in analysis.'' With David Drasin and Richard S. Laugesen. New Mathematical Monographs, 36. Cambridge University Press, Cambridge, 2019. xviii+473 pp.

\bibitem{Orig Schauder Estimates} L. Bers, L. Nirenberg, \emph{On linear and non-linear elliptic boundary value problems in the plane}, Convegno Internazionale sulle Equazioni Lineari alle Derivate Parziali, Trieste, 1954, pp. 141--167. Edizioni Cremonese, Roma, 1955.

\bibitem{Bossel}
M.-H. Bossel, \emph{Membranes \'elastiquement li\'ees: Extension du th\'eor\'eme de Rayleigh-Faber-Krahn
et de l'in\'egalit\'e de Cheeger}, C. R. Acad. Sci. Paris S\'er. I Math. 302 (1986), 47--50.

\bibitem{Periodic Steiner}F. Brock, \emph{Steiner symmetrization
and periodic solutions of boundary value problems}, Z. Anal. Anwendungen
13 (1994), no. 3, 417\textendash{}423.

\bibitem{BFNT}
D. Bucur, V. Ferone, C. Nitsch, and C. Trombetti, \emph{A sharp estimate for the first Robin-Laplacian
eigenvalue with negative boundary parameter}, preprint. Retrievable from https://arxiv.org/abs/1810.06108.

\bibitem{BucurGiacomini1}
D. Bucur and A. Giacomini, \emph{A variational approach to the isoperimetric inequality for the Robin
eigenvalue problem}, Arch. Ration. Mech. Anal. 198 (2010), 927--961.

\bibitem{BucurGiacomini2}
D. Bucur and A. Giacomini, \emph{Faber-Krahn inequalities for the Robin-Laplacian: a free discontinuity approach}, Arch. Ration. Mech. Anal. 218 (2015), 757--824.

\bibitem{ChasmanLangford}
L. M. Chasman and J. J. Langford, \emph{A sharp isoperimetric inequality for the second eigenvalue of the Robin plate}, preprint. Retrievable from https://arxiv.org/abs/2010.10576.

\bibitem{Daners}
D. Daners, \emph{A Faber-Krahn inequality for Robin problems in any space dimension}, Math. Ann.
335 (2006), 767--785.

\bibitem{EvansGariepy}
L. C. Evans and R. F. Gariepy, Measure theory and fine properties of functions. Studies in Advanced Mathematics. CRC Press, Boca Raton, FL, 1992. viii+268 pp.

\bibitem{FK}
P. Freitas and D. Krej\v{c}i\v{r}\'{\i}k, \emph{The first Robin eigenvalue with negative boundary parameter}, Adv.
Math. 280 (2015), 322--339.

\bibitem{FreitasLaugesen1}
P. Freitas and R. S. Laugesen,  \emph{From Neumann to Steklov and beyond, via Robin: the Weinberger way}, American Journal of Mathematics, to appear. Retrievable from https://arxiv.org/abs/1810.07461.

\bibitem{FrietasLaugesen2}
P. Freitas and R. S. Laugesen, \emph{From Steklov to Neumann and beyond, via Robin: the Szeg\H{o} way}, Canadian Journal of Mathematics, 72(4), 1024--1043.

\bibitem{GilbargTrudinger}
D. Gilbarg and N. S. Trudinger, ``Elliptic partial differential equations of second order.'' Reprint of the 1998 edition. Classics in Mathematics. Springer-Verlag, Berlin, 2001.

\bibitem{GirourdLaugesen}
A. Girouard and R. S. Laugesen. \emph{Robin spectrum: two disks maximize the third eigenvalue}, Indiana University Mathematics Journal, to appear.

\bibitem{Kawohl1}B. Kawohl, ``Rearrangements and Convexity of
Level Sets in PDE,'' Lecture Notes in Math. 1150, Springer, Berlin,
1985.

\bibitem{Olga}O. A. Ladyzenskaja and N. N. Ural'ceva, ``Linear
and Quasilinear Elliptic Equations'', Academic Press, New York-London,
1968.

\bibitem{LangfordThesis} J. J. Langford, ``Comparison Theorems in Elliptic Partial Differential Equations with Neumann Boundary Conditions,'' Thesis (Ph.D.)--Washington University in St. Louis. 2012.

\bibitem{Langford1} J. J. Langford, \emph{Symmetrization of Poisson's equation with Neumann boundary conditions}, Ann. Sc. Norm. Super. Pisa Cl. Sci. (5) 14 (2015), no. 4, 1025--1063.

\bibitem{Langford2}
J. J. Langford, \emph{Neumann comparison theorems in elliptic PDEs}, Potential Anal. 43 (2015), no. 3, 415--459.

\bibitem{Langford3}
J. J. Langford, \emph{Subharmonicity, comparison results, and temperature gaps in cylindrical domains}, Differential Integral Equations 29 (2016), no. 5-6, 493--512.

\bibitem{Leib and Loss Analysis}E. H. Lieb and M. Loss, ``Analysis'',
Second edition. Graduate Studies in Mathematics, 14. American Mathematical
Society, Providence, RI, 2001.

\bibitem{Littman}W. Littman, \emph{A strong maximum principle
for weakly L-subharmonic functions}, J. Math. Mech. 8 1959 761--770.

\bibitem{Nardi} G. Nardi, \emph{Schauder
estimate for solutions of Poisson's equation with Neumann boundary
conditions}, Enseign. Math. 60 (2014), no. 3-4, 421--435.

\bibitem{Sarvas}
J. Sarvas, \emph{Symmetrization of condensers in n-space}, Ann. Acad. Sci. Fenn. Ser. A. I. 1972, no. 522, 44 pp.

\bibitem{Talenti}G. Talenti, \emph{Elliptic equations and rearrangements,}
Ann. Scuola Norm. Sup. Pisa Cl. Sci (4) 3 (1976), 697--718.

\bibitem{TalentiSurvey}
G. Talenti, \emph{The art of rearranging}, Milan J. Math. 84 (2016), no. 1, 105--157.

\bibitem{Weitsman}A. Weitsman, \emph{Spherical symmetrization
in the theory of elliptic partial differential equations}, Comm. Partial
Differential Equations 8 (1983), no. 5, 545--561.

\end{thebibliography}
\end{document}